\documentclass[12pt,a4paper]{amsart}

\usepackage[english]{babel}
\usepackage[latin1]{inputenc}
\usepackage{amsmath,amssymb,amsfonts,amscd,amsthm}
\usepackage{enumerate,relsize,stmaryrd}
\usepackage[english]{babel}
\usepackage{color}
\usepackage{fullpage}

\usepackage{hyperref}
\usepackage{stmaryrd}

\newcommand{\Ker}{\operatorname{Ker}}

\newcommand{\smallcup}{\operatorname{\smallsmile}}

\newcommand{\image}{\operatorname{Im}}

\newcommand{\N}{\mathbb{N}}

\newcommand{\R}{\mathbb{R}}
\newcommand{\Z}{\mathbb{Z}}
\newcommand{\C}{\mathbb{C}}
\newcommand{\Ad}{\operatorname{Ad}}
\newcommand{\ad}{\operatorname{ad}}

\newcommand{\tr}{\operatorname{tr}}
\newcommand{\SLn}[1][n]{\mathrm{SL}(#1)}
\newcommand{\sln}[1][n]{\mathfrak{sl}(#1)}
\newcommand{\GLn}[1][n]{\mathrm{GL}(#1)}
\newcommand{\gln}[1][n]{\mathfrak{gl}(#1)}

\newcommand{\Spec}{\operatorname{Spm}}

\def\co{\colon\thinspace}
\newtheorem{thm}{Theorem}[section]
\newtheorem{prop}[thm]{Proposition}
\newtheorem{lemma}[thm]{Lemma}

\newtheorem{cor}[thm]{Corollary}

\theoremstyle{definition}
       
       \newtheorem{remark}[thm]{Remark}
       \newtheorem{example}[thm]{Example}

\title{On the local structure of the $\mathrm{SL}(n,\mathbb{C})$ representation variety of knot groups}

\author[L. Ben Abdelghani]{Leila Ben Abdelghani}
\address{Laboratory of  Analysis, Probability and fractals, University of Monastir, 
Boulevard de l'environnement, 5019 Monastir, Tunisia}
\email{leila.benabdelghani@fsm.rnu.tn}

\author[M. Heusener]{Michael Heusener}
\address{
Universit\'e Clermont Auvergne, CNRS, Laboratoire de Math\'ematiques Blaise Pascal, F-63000 Clermont-Ferrand, France.
} 
\email{michael.heusener@uca.fr}

\begin{document}

\selectlanguage{english}

\begin{abstract}
We study the local structure of the representation variety of a knot group into the special linear group of degree $n$ over the complex numbers
at certain diagonal representations. In particular we determine the tangent cone of the representation variety at these diagonal representations, and show that the latter can be deformed into irreducible representations. 
Furthermore, we use Luna's slice theorem to analyze the local structure of the  character variety.
\end{abstract}



\maketitle

\section{Introduction }
\label{Introduction}

The aim of this paper is to analyze the local structure of the variety of representations (resp. of characters) of a knot group into the special linear group $\SLn[n,\C]$, $n\geq2$, in the neighborhood of  certain diagonal representations (resp. characters). It is a well known fact that if $\Gamma$ is a finitely generated group and $G$ is a linear algebraic  group then the set $R(\Gamma, G)$ of representations $\Gamma\to G$ has the natural structure of an affine algebraic scheme (see \cite{Lubotzky-Magid1985}). A first approximation to the local structure of the representation variety at a given representation $\rho$ is given by its Zariski tangent space. Two  finer approximations are the  \emph{quadratic} and the \emph{tangent cones}
(see works of W. Goldman and J. Millson \cite{Goldman1984,GoldmanMillson1988}).

Given a finitely generated group $\Gamma$, we denote by $R_n(\Gamma) := R(\Gamma, \SLn[n,\C])$ 
the $\SLn[n,\C]$-representation variety of $\Gamma$.
The group $\SLn[n,\C]$ acts on $R_n(\Gamma)$ by conjugation, and the algebraic quotient of this action is the variety of characters which will be denoted by $X_n(\Gamma) : =R_n(\Gamma)\sslash\SLn[n,\C]$.

In this article, we focus on fundamental groups  of exteriors of knots in homology spheres, and
determine explicitly the quadratic and tangent cones of the representation variety at certain diagonal representations.
For the fundamental group $\Gamma$  of the exterior of a knot $K$ in a homology sphere, we obtain a canonical surjection $h\co\Gamma\to\Z\cong\Gamma/\Gamma'$. 
The representation space $R_n(\Z)$ is isomorphic to $\SLn[n,\C]$, and the surjection $h$ 
induces a \emph{closed immersion} $R_n(\Z)\hookrightarrow R_n(\Gamma)$. In fact, $R_n(\Z)\subset R_n(\Gamma)$ forms 
an $(n^2-1)$--dimensional algebraic component of $R_n(\Gamma)$.
A particular representation $\rho_D\in R_n(\Gamma)$ is given by a \emph{regular} diagonal matrix
 $D=D(\lambda_1,\ldots,\lambda_n)\in\SLn[n,\C]$ i.e. all the diagonal elements are pairwise different. 
More precisely, we let $\rho_D\co\Gamma\to\SLn[n,\C]$ denote the diagonal representation which factors through
the abelianization $\Gamma / \Gamma'$, and which is given by $\rho_D(\gamma) = D^{h(\gamma)}$ (see Section~\ref{sec:notations} for details).
If none of the quotients $\lambda_i/\lambda_j$ is a root of the Alexander polynomial $\Delta_K$, then
$\rho_D$ is a smooth point of $R_n(\Z)$, and there exists a neighborhood of $\rho_D$ which consists only of representations with abelian images. 
On the other hand, if all quotients  $\alpha_i=\lambda_i/\lambda_{i+1}$, $i=1,\ldots,n-1$, are simple roots of the Alexander polynomial 
$\Delta_K$ and if $\Delta_K(\lambda_i/\lambda_j)\neq0$ for all $|i-j|\geq 2$, then there exists a representation $\rho^{(n)}_D$ into the solvable group of upper triangular matrices. The representation  $\rho^{(n)}_D\in R_n(\Gamma)$  is reducible and non-semisimple.
Its existence was motivated by a result of G.~Burde \cite{Burde} and G.~de~Rham \cite{deRham}.
The orbit of the  representation $\rho^{(n)}_D$ under the action of conjugation of $\SLn[n,\C]$ is not closed, but it is of maximal dimension and its closure contains the diagonal representation $\rho_D$.

We obtain our main result which  generalizes \cite[Cor.~1.3]{Heusener-Porti-Suarez2001}, and \cite[Thm.~1.2]{Heusener-Porti2005}. 

\begin{thm}\label{thm:main1} Let $K$ be a knot in a three-dimensional integer homology sphere, and let 
$D=D(\lambda_1,\ldots,\lambda_n)\in\SLn[n,\C]$ be a regular diagonal matrix.
If all quotients  $\alpha_i=\lambda_i/\lambda_{i+1}$, $i=1,\ldots,n-1$, are simple roots of the Alexander polynomial 
$\Delta_K$ and if $\Delta_K(\lambda_i/\lambda_j)\neq0$ for all $|i-j|\geq 2$, then 
there exists a unique algebraic component $R^{(n)}_D$ of $R_n(\Gamma)$ which passes through $\rho_D$ and which contains irreducible representations.

Moreover, $\dim R^{(n)}_D = (n^2+n-2)$, 
$\rho_D$ is reduced, smooth point of $R^{(n)}_D$.
The components $R_n(\Z)$ and $R^{(n)}_D$ intersect transversally at the orbit  of $\rho_D$.
\end{thm}
The first part of Theorem~\ref{thm:main1} is proved in Section~\ref{sec:main}.
and the assertion about $\rho_D$ follows from Proposition~\ref{prop:R^n}.

\smallskip
For each representation $\rho\co\Gamma\to\SLn[n,\C]$ the Lie algebra $\sln[n,\C]$ turns into a $\Gamma$-module via $\Ad\rho$.
Due to a result of A. Weil \cite{Weil1964}, the Zariski tangent space at a representation $\rho\in R_n(\Gamma)$ 
is contained in the space of  crossed morphisms or $1$-cocycles $Z^1(\Gamma; \Ad\rho)$.
We call a cocycle $U\in Z^1(\Gamma; \Ad\rho)$  \emph{integrable} if there exists an analytic path
$\rho_t$ in $R_n(\Gamma)$ such that $\rho_0 = \rho$, and which is tangent to $U$ i.e. for all $\gamma\in\Gamma$ we have:
\[
U(\gamma) =\frac{d\rho_t(\gamma)}{dt}\Big|_{t=0}\rho(\gamma)^{-1}\,.
\]
By work of W. Goldman \cite{Goldman1984}, the first obstruction to integrability of the cocycle $U$ is given by the second cohomology class represented by its cup product  $[U\smallcup U]$. We prove that for the diagonal representation $\rho_D$ and under the hypothesis of Theorem~\ref{thm:main1} this obstruction is sufficient:
\begin{thm}\label{thm:main3}
Under the hypothesis of Theorem~\ref{thm:main1}, 
a tangent vector $U \in Z^1(\Gamma; \Ad\rho_D)$ is integrable if and only if its cup product  $[U\smallcup U]$ 
represents the trivial cohomology class in $H^2(\Gamma;\Ad\rho_D)$.
\end{thm}

Moreover, under the hypothesis of Theorem~\ref{thm:main1}, we will show (see Section \ref{sec:diagonal}) that the \emph{tangent cone} 
$TC_{\rho_D} (R_n(\Gamma))$ at $\rho_D$ is the union of $2^{n-1}$ affine subspaces each of dimension $n^2-1+k$; $0\leq k\leq n-1$.
Thus,  for each $k$, $0\leq k \leq n-1$, there is at least one
irreducible algebraic component of $R_n(\Gamma)$ of dimension $n^2-1+k$ which contains $\rho_D$. It follows that for $n\geq2$
the representation $\rho_D$ is contained in at least $n$ 
components 
(see Subsection~\ref{sec:number_components}).

The component $R_n(\Z)$ is the unique component of dimension $(n^2-1)$, 
and it contains only abelian representations.
The component $R_D^{(n)}$ is the unique component of dimension 
$(n^2+n-2)$, and it is the unique component containing irreducible representations. 
All the other components contain only reducible representations and their dimension ranges between 
$n^2$ and $n^2+n-3$ (see Theorem~\ref{thm:components}).

\medskip

Let $t\co R_n(\Gamma)\to X_n(\Gamma)$ be the canonical projection and denote by $X_n(\Z):=t(R_n(\Z))$ and 
$X^{(n)}_D:=t(R^{(n)}_D)$  the quotients of the components $R_n(\Z)$ and $R^{(n)}_D$ respectively.
In fact $t\co R_n(\Z)\to X_n(\Z)\cong \C^{n-1}$ maps the matrix $A\in \SLn[n,\C]$ onto the coefficients of the characteristic polynomial of $A$ (see \cite[Example 1.2]{dolgachev2003lectures}).

As the set $R_n^{irr}(\Gamma)$ of irreducible representations in $R_n(\Gamma)$ is Zariski open and invariant by conjugation, 
its image $X_n^{irr}(\Gamma)=t(R_n^{irr}(\Gamma)) $ is Zariki open in $X_n(\Gamma)$. Their Zariski closures in $R_n(\Gamma)$ and $X_n(\Gamma)$ will be denoted by  $\overline{R_n^{irr}(\Gamma)}$ and $\overline{X_n^{irr}(\Gamma)}$ respectively.

\begin{thm}\label{thm:main2}
Under the hypothesis of Theorem~\ref{thm:main1} we obtain:
\begin{enumerate}
\item The character $\chi_D$ is  a smooth point of each of the two $(n-1)$-dimensional components $X_n(\Z)$ and $X_D^{(n)}$.

   \item The character $\chi_D=t(\rho_D)$ is a smooth point of $\overline{X_n^{irr}(\Gamma)}$.

\item The intersection of the two above components at $\chi_D$ is trivial:
\[
T_{\chi_D}^{Zar}X_n(\Z)\cap T_{\chi_D}^{Zar}X_D^{(n)}=\{0\}\,.
\]
\item Let $\ad\co\sln[n,\C]\to T_{\rho^{(n)}_D}^{Zar}R_n(\Gamma)$ and $dt\co T_{\rho^{(n)}_D}^{Zar}R_n(\Gamma)\to T_{\chi_D}^{Zar}X_n(\Gamma)$ be the tangent morphisms induced by the conjugation action of $\SLn[n,\C]$ on $\rho^{(n)}_D$ and by the canonical projection $t\co R_n(\Gamma)\to X_n(\Gamma)$.

Then we have the exact sequence:
\[
0\to\sln[n,\C]\stackrel{ad}{\longrightarrow} T_{\rho^{(n)}_D}^{Zar}R_n(\Gamma) \stackrel{dt}{\longrightarrow} T_{\chi_D}^{Zar}\overline{X_n^{irr}(\Gamma)}\to0\,,
\]
and $H^1(\Gamma;\Ad\rho_D^{(n)})\cong T_{\chi_D}^{Zar}\overline{X_n^{irr}(\Gamma)}$.

\end{enumerate}
\end{thm}

\paragraph{Plan of the article.} In Section~\ref{sec:notations} we recall some results that we will need in our study, and we introduce the necessary notations. The experienced reader can skip this section and use it as a reference. In Section~\ref{sec:main} we construct the non-semisimple representation  $\rho^{(n)}_D$, and we show that it is a smooth point of $R_n(\Gamma)$ contained in a unique component $R^{(n)}_D$ of dimension $n^2+n-2$. Furthermore, we prove that the component $R^{(n)}_D$ contains irreducible representations. In Section~\ref{sec:diagonal} we determine the quadratic cone of the diagonal representation $\rho_D$ and  prove that it coincide with the tangent cone. 
Various examples are given in Section~\ref{sec:examples}. Finally, in Section~\ref{sec:character} we use our results to analyze the local structure of the character variety.

\section{Notations and background results}
\label{sec:notations}
 For the convenience of the reader, we state in this section the definitions and results that will be used in what follows. 
 More precisely, Subsection~\ref{Cohomological results} is devoted to collect results and definitions from group cohomology. In Subsection~\ref{sec:useful} we review some general facts about affine algebraic schemes, Zariski tangent space, quadratic and tangent cones. 
In Subsection~\ref{sec:curvetangents} we 
  recall the link between the quadratic cone and the curve tangents up to second order.
  In Subsection~\ref{sec:rep-scheme} we focus on the representation schemes and varieties. 
 Finally, in Subsection~\ref{Alexander_module} we recall some facts about the Alexander module, as well as certain cohomological results.
  
  \subsection{Group cohomology}
\label{Cohomological results}

The general reference for group cohomology is K.~Brown's book \cite[Chap.~III]{Brown1982}.

Let $A$ be a $\Gamma$-module. We denote by $C^*(\Gamma;A)$ the cochain complex and by $\delta\co C^n(\Gamma;A)\to C^{n+1}(\Gamma;A)$  the coboundary operator. 
The coboundaries (respectively cocycles, cohomology) of $\Gamma$ with coefficients in $A$ are denoted by $B^*(\Gamma;A)$ (respectively $Z^*(\Gamma;A),\ H^*(\Gamma;A)$.  
In what follows
$1$-cocycles and $1$-coboundaries  will be also called \emph{derivations} and \emph{principal derivations} respectively.

A short exact sequence:
\[ 0\to A_1 \stackrel{i}{\longrightarrow} A_2 \stackrel{p}{\longrightarrow} A_3 \to 0\]
of $\Gamma$-modules gives rise to a short exact sequence of cochain complexes:
\[ 0 \to C^*(\Gamma;A_1)\stackrel{i^*}{\longrightarrow} C^*(\Gamma;A_2) \stackrel{p^*}{\longrightarrow} C^*(\Gamma;A_3) \to 0\,.\]
We will make use of the corresponding long exact cohomology sequence
(see \cite[III. Proposition~6.1]{Brown1982}):
\[ 0\to H^0(\Gamma;A_1)\longrightarrow H^0(\Gamma;A_2) \longrightarrow H^0(\Gamma;A_3)\stackrel{\beta^0}{\longrightarrow}
H^1(\Gamma;A_1)\longrightarrow \cdots\]
Recall that the Bockstein homomorphism
$\beta^{n}\co H^n(\Gamma;A_3)\to H^{n+1}(\Gamma;A_1)$
is determined by the snake lemma:
if $z\in Z^n(\Gamma;A_3)$ is a cocycle and if $\tilde z \in (p^*)^{-1}(z)\subset C^n(\Gamma;A_2)$ is any lift of $z$ then  $\delta_2 (\tilde z) \in \image(i^*)$ where $\delta_2$ denotes the coboundary operator of $C^*(\Gamma;A_2)$. It follows that any cochain $z'\in C^{n+1}(\Gamma;A_1)$ such that
$i^*(z') = \delta_2 (\tilde z)$ is a cocycle and that its cohomology class does only depend on the cohomology class represented by $z$. The cocycle $z'$ represents
 the image of the cohomology class represented by $z$ under $\beta^{n}$.

 By abuse of notation and if no confusion can arise,  we will sometimes identify a cocycle with its cohomology class and will write
 $\beta^n(z)$ for a cocycle $z\in Z^n(\Gamma;A_3)$ even if the map $\beta^n$ is only well defined on cohomology classes. This will simplify the notations.
 
Let $A_1,\ A_2$ and $A_3$ be $\Gamma$-modules. The cup product of two cochains $u_1\in C^p(\Gamma;A_1)$ and $u_2\in C^q(\Gamma;A_2)$ is the cochain $u_1\smallcup u_2\in
C^{p+q}(\Gamma; A_1\otimes A_2)$ defined by:
\[
u_1\smallcup u_2(\gamma_1,\ldots,\gamma_{p+q}):=u_1(\gamma_1,\ldots,\gamma_p)\otimes\gamma_1\cdots\gamma_p\cdot u_2(\gamma_{p+1},\ldots,\gamma_{p+q})\,.
\]
Here $A_1\otimes A_2$ is a $\Gamma$-module via the diagonal action.

The cup-product is compatible with the boundary operator in the following sense:
\[
\delta (u_1\smallcup u_2) = \delta(u_1)\smallcup u_2 + (-1)^p u_1\smallcup\delta(u_2)\,.
\]

When combined with any $\Gamma$-invariant bilinear map $A_1\otimes A_2\to A_3$, this gives a cup product: 
\begin{equation}\label{eq:cup}
\smallcup^b\co C^1(\Gamma;A_1)\times C^1(\Gamma;A_2) \to C^2(\Gamma;A_3).
\end{equation}

In this paper, we will mainly make use of two bilinear forms: the multiplication of complex numbers and the Lie bracket when $A=\mathfrak g$ is a Lie algebra.
We will denote, respectively, $\cdot\smallcup\cdot$ and $[\cdot\smallcup\cdot]$ the corresponding cup products.

Let $b\co A_1\otimes A_2\to A_3$ be bilinear and let $z_i\in  Z^1(\Gamma; A_i)$, $i = 1, 2$, be cocycles.
We define $f\co \Gamma\to A_3$ by $f(\gamma) := b(z_1(\gamma) \otimes z_2(\gamma))$. A direct calculation gives:
$$\delta f(\gamma_1,\gamma_2) + b(z_1(\gamma_1) \otimes \gamma_1 \cdot z_2(\gamma_2)) + b(\gamma_1 \cdot z_1(\gamma_2) \otimes z_2(\gamma_1)) = 0.$$ This shows
that: 
\begin{equation}\label{eq:cupbil}
\delta f + z_1 \smallcup^b z_2 + z_2\smallcup^{b\circ \tau} z_1 = 0
\end{equation}
where $\tau\co A_1 \otimes A_2 \to A_2 \otimes A_1$ is the twist operator.

 \subsection{Affine algebraic schemes}
\label{sec:useful}

This section can be skipped by readers familiar with schemes and can be used as a reference.
For preliminaries to  affine algebraic schemes we follow the development in Appendix~A in \cite{Milne}. For a summary of basic results  the reader is invited to consult \cite[Section~2.1]{HP23}. %
General references for schemes include the books of D.~Eisenbud and J.~Harris \cite{EisenbudHarris},
E.~Kunz \cite{Kunz} and 
D.~Mumford \cite{Mumford-redBook}.

We are only interested in subschemes of a complex affine space
$\mathbb{A}^N$. The coordinate algebra of $\mathbb{A}^N$ will be denoted by $S$ and it is isomorphic to the polynomial ring $\C[x_1,\ldots,x_N]$. A subscheme $X$ of $\mathbb{A}^N$ is given by an ideal $J := J(X) \subset S$. Formally, an affine subscheme  is a pair 
$X= ( \Spec(S/J),S/J)$ where $|X|:= \Spec(S/J)$ is the maximal spectrum of the $\C$-algebra
$S/J$ which carries the Zariski topology. The quotient $S/J$ is called the \emph{coordinate algebra} or the \emph{algebra of regular functions} of $X$, and will be denoted by $\mathcal{O}(X)$.
A  point $p$ of $|X|$ corresponds to a maximal ideal $\mathfrak{m}_p$ in $\mathcal{O}(X)$, and the value of the regular function 
$f\in \mathcal{O}(X)$ at $p$ is the reduction $f(p) := f\bmod\mathfrak{m}_p$.
For any ideal $I\subset\mathcal{O}(X)$ we define 
$V(I) =\{ p\in|X|\mid I\subset \mathfrak m_p\}$, and we can define the subscheme $Y = (V(I), \mathcal{O}(X)/I)$ (see \cite[Section I.2.1]{EisenbudHarris}).

Let $X\subset \mathbb{A}^N$ be an affine subscheme and $p\in|X|$ a point.
We denote $ J := J(X)$ the corresponding ideal of $S$.
In the definition of the
tangent space $T_p(X)$ 
one takes all $f \in J$, expanded around $p$ and take their linear parts. It is sufficient to take only the linear part of a finite set of generators of $J$.
The tangent space $T_p(X)$ is the subscheme  corresponding to the ideal 
$J^{(1)}_p$ generated by 
these homogeneous linear forms, i.e.\ $T_p(X)=V(J^{(1)}_p)$. In fact it is an affine subspace of 
$\mathbb{A}^N$. 

The difference with the definition of the tangent cone
$TC_p(X)$ to $X$ at $p$ is this: we take all $f \in J$, expand around $p$, and look not at their linear terms but at their \emph{leading} terms, all these leading terms generate an ideal $J^{(\infty)}_p$. Notice that the leading terms of a set of generators of $J$ is in general not sufficient.

Similar, for the quadratic cone $TQ_p(X)$ we take all $f \in J$, expand around $p$, and look at their leading term if its degree is at most two, all these leading terms generate an ideal $J^{(2)}_p$. We obtain: 
\[
 J^{(1)}_p \subset J^{(2)}_p \subset J^{(\infty)}_p\quad\text{ and }\quad
 T_p(X)\supset TQ_p(X) \supset TC_p(X) \,.
\]
Notice that the tangent space is the smallest affine subspace of $\mathbb{A}^N$ which contains the tangent cone $TC_p(X)$ (see \cite[III, \S4]{Mumford-redBook} for more details).

Moreover, for each point $p\in |X|$ we have the following inequalities:
\begin{equation}
\label{eqn:ienequalities}
\dim_p X=\dim TC_p(X)\leq \dim TQ_p(X) \leq \dim T_p (X)\,.
\end{equation}
The point $p$ is called \emph{nonsingular} or \emph{smooth} if the tangent cone equals the tangent space $TC_p(X)=T_p(X)$. Notice that  $\dim TC_p(X)= \dim T_p(X)$ implies that $p$ is nonsingular.

An algebra is called \emph{reduced} if it has no non-zero nilpotent elements.
A scheme $X$ is called \emph{reduced} if its coordinate algebra $\mathcal{O}(X)$ is reduced, 
and a point $p$ of $X$ is called \emph{reduced} if its local ring $\mathcal{O}(X)_{\mathfrak m_p}$ is reduced.

In what follows we will give a sufficient condition on a  point $p\in X$ to be reduced. More precisely, we will make  use of the fact that the coordinate algebra of $TC_p(X)$ is isomorphic to the associated graded algebra of the local ring of $p\in X$ (see \cite[III.\S3]{Mumford-redBook}, \cite[Chap.~6, Prop.~1.2]{Kunz} and \cite{HP23}).

\begin{lemma}\label{lem:scheme2}
Let  $X=\Spec (k[X_1,\ldots,X_N]/I)\subset \mathbb{A}^N$ be an affine scheme, and
let $p\in X$ be a point. If the tangent cone $TC_p(X)$ is reduced then the local ring
$\mathcal{O}(X)_{\mathfrak m_p}$ is also reduced and thus $p$ is reduced.
\end{lemma}

If $X$ is a scheme, there exists a unique reduced subscheme $X_\mathit{red}$ of $X$ with reduced coordinate algebra $\mathcal{O}(X_\mathit{red}) = \mathcal{O}(X)/N$ (here
$N$ is the nilradical of $\mathcal{O}(X)$, see \cite{EisenbudHarris,HP23}). 
The topological spaces $|X|$ and $|X_\mathit{red}|$ are homeomorphic.
Notice that if $p$ is a reduced point then the tangent spaces $T_p (X_\mathit{red})$ and $T_p (X)$ coincide.

\subsection{Curve tangents and quadratic cone}
\label{sec:curvetangents}

We recall in this section the link between the quadratic cone and the curve tangents up to second order.
In what follows we consider the Artin algebras $A_n := \C[t]/(t^{n+1})$, $n\in\N$. The algebra $A_0$ is identified with $\C$, and $A_1$ is commonly identified with dual numbers $\C[\epsilon]$, $\epsilon^2=0$. 
We let $\pi_n\co A_{n+1}\to A_n$ denote the canonical projection. 

As in Subection~\ref{sec:useful}, we let $X\subset \mathbb{A}^n$ denote an affine algebraic subscheme with coordinate ring $\mathcal{O}(X)$. For every point $p\in X$ we let $\mathfrak{m}_p\subset\mathcal{O}(X)$ denote the corresponding maximal ideal and $\alpha_p\co \mathcal{O}(X)\to \C$ the corresponding algebra morphism with $\mathrm{Ker}(\alpha_p)=\mathfrak{m}_p$. 
In other words, $\alpha_p(f) = f(p)$. A tangent vector $v_p\in T_pX$ corresponds to a morphism $\alpha_{v_p}\co\mathcal{O}(X) \to A_1$, given by $\alpha_{v_p}(f) = f(p) + \epsilon\, df(p)(v_p)$, such that $\pi_0\circ\alpha_{v_p} =\alpha_p$. We say that $v_p\in T_pX$ is a \emph{curve tangent of order two at $p$} if there exists a morphism $\alpha_2\co\mathcal{O}(X)\to A_2$ such that $\pi_1\circ\alpha_2 = \alpha_v$.

\begin{lemma} \label{lem:2nd_curve_tangent} Let $p\in X$ be a point. The space of curve tangents of order two at $p$ has the structure of an affine algebraic subscheme of $T_pX$. More precisely, it coincides with the quadratic cone of $X$ at $p$.
\end{lemma}
\begin{proof}
The coordinate ring of $X\subset\mathbb{A}^n$ has a presentation 
$\C[x_1,\ldots,x_n]/(f_1,\ldots,f_m)$.
After an affine change of coordinates, and a change of generators of the ideal $J=(f_1,\ldots,f_m)$, we can assume that $\mathfrak{m}_p=(x_1,\ldots,x_n)\supseteq J$, 
\[
f_i = x_i + f_i^{(2)} \bmod \mathfrak{m}_p^3\quad\text{ for $1\leq i \leq r$, and }
f_i = f_i^{(2)} \bmod \mathfrak{m}_p^3\ \quad\text{ for $r<i\leq m$.}
\]
Here $r=n-\dim T_pX$ is the rank of the linearization of the $f_i$'s, and $f_i^{(2)}$ is the second homogenous component of $f_i = f_i^{(1)} + f_i^{(2)} \bmod \mathfrak{m}_p^3$.

The tangent space $T_pX$ is then given by the equations $x_i=0$, $1\leq i\leq r$. 
Hence a substitution $x_i =a_i t$, $a_i\in \C$, defines an algebra morphism
\[
\alpha_v\co\C[x_1,\ldots,x_n]/(f_1,\ldots,f_m)\to A_1
\]
if and only if $a_i=0$, $1\leq i\leq r$. 
The corresponding tangent vector is $v=(a_{r+1},\ldots,a_n)\in T_pX$. Now,
$\alpha_v$ extends to a morphism $\alpha_2\co \mathcal{O}(X)\to A_2$ which is defined by the substitution 
$x_i =a_i t + b_i t^2$, $b_i\in \C$, if and only if
\[
b_i + f_i^{(2)}(0,\ldots,0,a_{r+1},\ldots,a_n) =0 \quad \text{ for $1\leq i \leq r$,} 
\]
and
\[
f_i^{(2)}(0,\ldots,0,a_{r+1},\ldots,a_n) =0 \quad \text{ for $r+1\leq i \leq m$.}
\]
Hence the coordinate algebra of second order curve tangents at $p$ is isomorphic to 
\begin{equation}\label{eq:curve_alg} 
\C[t_1,\ldots,t_{n-r}]/(g_1,\ldots,g_{m-r})
\end{equation}
where 
$g_j(t_1,\ldots,t_{n-r}) = f_{r+j}^{(2)}(0,\ldots,0,t_1,\ldots,t_{n-r})$, $1\leq j \leq m-r$,
are homogenous polynomials of degree two.

The ideal of the quadratic cone  $J^{(2)}$ is the ideal generated by the leading terms of polynomials of $J$ of degree lower or equal to two.
We will show that $J^{(2)}$ coincides with the ideal $K$ which is generated by the leading terms of the generators of $J$ i.e.\ $J^{(2)} = K = (x_1,\ldots,x_r,f_{r+1}^{(2)},\ldots,f_m^{(2)})$.
From this it follows directly that the algebra of the quadratic cone is isomorphic to the algebra defined in equation~\eqref{eq:curve_alg} i.e.
\[
\C[x_1,\ldots,x_n]/ J^{(2)} \cong \C[t_1,\ldots,t_{n-r}]/(g_1,\ldots,g_{m-r})\,.
\]

It is clear that $K\subset J^{(2)}$. In order to prove that $K\supset J^{(2)}$ we consider a polynomial $g\in J$,
\(
g = \sum_{i=1}^m a_i f_i 
\) with polynomials $a_i\in \C[x_1,\ldots,x_n]$. Now a direct calculation shows that
\[
g^{(1)} = \sum_{i=1}^r a_i^{(0)} f_{i}^{(1)}
\quad \text{ and }\quad
g^{(2)} = \sum_{i=1}^m a_i^{(0)} f_{i}^{(2)} + \sum_{i=1}^r a_i^{(1)} x_i\,.
\]
It follows from the first equation that the linear form $g^{(1)}$ is a linear combination of $x_1,\ldots,x_r$.
Observe that the leading term of $g^{(1)}$ is zero if and only if $a_i^{(0)}=0$, $1\leq i\leq r$.
In this case the leading term of $g$ is $g^{(2)} =  \sum_{j=r+1}^m a_j^{(0)} f_{j}^{(2)} + \sum_{i=1}^r a_i^{(1)} x_i$ which is contained in the ideal~$K$.
\end{proof}

\subsection{Representation scheme and variety}
\label{sec:rep-scheme}

The principal reference for this section is Lubotzky's and Magid's book \cite{Lubotzky-Magid1985}.
To shorten notation we will simply write $\SLn$ and $\GLn$ instead of $\SLn[n,\C]$ and $\GLn[n,\C]$ respectively.
The same notation applies for the Lie algebras
$\sln=\sln[n,\C]$ and $\gln=\gln[n,\C]$.

Let $\Gamma$ be a finitely generated group. The set of all homomorphisms of $\Gamma$ into $\SLn$ has the structure of an \emph{affine scheme}. 
More precisely,  the
$\SLn$ \emph{scheme of representations}  $ R_n(\Gamma):= R(\Gamma, \SLn)$ is defined as the \emph{spectrum} 
of the universal algebra $A_n(\Gamma)$:
\[
R_n(\Gamma) := \Spec(A_n(\Gamma))
\]
(see \cite{Lubotzky-Magid1985} or \cite{HP23} for more details).
The scheme $R_n(\Gamma) $ is equipped with the  \emph{Zariski topology}.
A single representation $\rho\colon\Gamma\to\SLn$ corresponds to a $\C$-algebra morphism
$A_n(\Gamma)\to \C$ i.e.\ to a maximal ideal $\mathfrak m_\rho \subset A_n(\Gamma)$, and each maximal ideal $\mathfrak m\subset A_n(\Gamma)$ determines a representation
$\rho_\mathfrak{m}\colon\Gamma\to\SLn$ (this follows from Zariski's Lemma).
Hence a maximal ideal in $A_n(\Gamma)$ corresponds exactly to a representation
$\Gamma\to\SLn$, and we will use this identification in what follows. 

Notice that the universal algebra $A_n(\Gamma)$ might be not reduced. Therefore,
one has to consider the underlying affine scheme  with a possible non-reduced coordinate ring.

For more details and a  presentation of these notions see also \cite[\S2]{Heusener2016}.

Associated to the representation scheme is the representation variety
$R_n(\Gamma)_\mathit{red}$ with reduced coordinate ring. The representation variety
$R_n(\Gamma)_\mathit{red}$ is a closed subscheme of $R_n(\Gamma)$.
The underlying topological spaces of $R_n(\Gamma)$ and $R_n(\Gamma)_\mathit{red}$ are homeomorphic and hence $\dim R_n(\Gamma) = \dim R_n(\Gamma)_\mathit{red}$.

Let $\rho\co\Gamma\to \SLn$ be a representation.
The Lie algebra $\sln$ turns into a $\Gamma$-module via
$\Ad\rho$. 
 
A \emph{$1$-cocycle} or \emph{derivation} $d\in Z^1(\Gamma;\Ad\rho)$ is a map
$d\co\Gamma\to \sln$ satisfying
\[d(\gamma_1\gamma_2)=d(\gamma_1)+\Ad_{\rho(\gamma_1)}(d(\gamma_2))\quad\ \forall\ \gamma_1,\ \gamma_2\in\Gamma\,.\]

It was observed by A. Weil \cite{Weil1964} that the tangent space
$T_\rho R_n(\Gamma)_\mathit{red}$ of the representation variety at $\rho$ embeds into $Z^1(\Gamma;\Ad\rho)$. More precisely, there is
an isomorphism between the tangent space of the representation scheme
$T_\rho R_n(\Gamma)$ at $\rho$ and the space of 1-cocycles $Z^1(\Gamma;\Ad\rho)$ (see \cite[Proposition~2.2]{Lubotzky-Magid1985}):
\[
T_\rho R_n(\Gamma)\cong Z^1(\Gamma;\Ad\rho)\,.
\]

 Informally speaking, given a smooth curve $\rho_t$ of representations through $\rho_0=\rho$ one gets a $1$-cocycle $d\co\Gamma\to \sln$ by defining:
\[ d(\gamma) := \left.\frac{d \, \rho_{t}(\gamma)}
{d\,t}\right|_{t=0} \rho(\gamma)^{-1},
\quad\forall\gamma\in\Gamma\,.\]

We obtain the following inequalities:
\[
\dim_\rho R_n(\Gamma)=\dim_\rho R_n(\Gamma)_\mathit{red} \leq  \dim T_\rho R_n(\Gamma)_\mathit{red}\leq
\dim T_\rho R_n(\Gamma) = \dim Z^1(\Gamma;\Ad\rho)\;.
\]

The two inequalities are independent and if
$\dim_\rho R_n(\Gamma) =  \dim T_\rho R_n(\Gamma)_\mathit{red}$ then $\rho$ is a smooth point of the representation variety $R_n(\Gamma)_\mathit{red}$.
We call $\rho\in R_n(\Gamma)$ \emph{scheme smooth} if
$\dim_\rho R_n(\Gamma) = \dim Z^1(\Gamma;\Ad\rho)$. The representation $\rho\in R_n(\Gamma)$ is called \emph{reduced} if the local ring $\mathcal{O}_\rho(R_n(\Gamma))$ is reduced.
In this case $\dim T_\rho R_n(\Gamma)_\mathit{red} = \dim T_\rho R_n(\Gamma)$ holds independent from the first inequality. Therefore, a reduced representation is scheme smooth if and only if it is a smooth point of the representation variety.

It is easy to see that the tangent space to the orbit by conjugation $ O(\rho)$ of $\rho$ corresponds to
the space of principal derivations $B^1(\Gamma;\Ad\rho)$. Here a principal derivation is a map $b\co\Gamma\to\sln$ for which there exists $X\in\sln$ such that for all $\gamma\in\Gamma$ we have $b(\gamma) = \Ad_{\rho(\gamma)} (X) - X$. A detailed account can be found in \cite{Lubotzky-Magid1985}.

\medskip

The following lemma will be useful in what follows.
\begin{lemma}\label{lem:reducedpoints}
Let $M$ denote a three-dimensional compact, orientable manifold with torus boundary $\partial M \cong S^1\times S^1$.
 Let $\rho\in R_n(\pi_1M)$ be a representation. We have
\begin {enumerate}
\item \label{red1} $\dim H^1(\pi_1 M,\Ad\rho) \geq n-1$ ;
\item \label{red2} Suppose that  $\rho$ is a reduced point of the representation scheme $R_n(\pi_1 M)$, and let $V$ be a component of the
representation variety $R_n(\pi_1 M)_{red}$ such that $\rho\in V$.
If $V$ contains a representation $\varrho$ such that $\dim H^0(M,\Ad\varrho)=k$ then we have $\dim V\geq n^2+n-2-k$.
\end{enumerate}
\end{lemma}
\begin{proof}

In order to prove {\it(\ref{red1})}, first notice that
 $H^1(\pi_1M; \Ad\rho)\cong H^1(M; \Ad\rho)$ (Lemma~3.1 in \cite{Heusener-Porti2005}). 
On the other hand,  it follows from the 
\emph{Half Lives Half Dies lemma} that $\dim H^1( M,\Ad\rho) \geq n-1$ (see Lemme 3.2 in \cite{Heusener-Porti2015}).

To prove {\it(\ref{red2})} assume that $V$ is a component of $R_n(\pi_1M)_{red}$ which contains the reduced representation $\rho$.
Then $V' := \overline{(R_n(\pi_1M)_{red}\smallsetminus V)} \cap V$ is a Zariski-closed subset of $V$ of smaller dimension than $V$. Hence,
$U' := V\smallsetminus V'$ is a Zariski-open non-empty subset of $V$.
Simple points of $V$ form also a Zariski-open non-empty subset $U$  of $V$.
Hence $U\cap U'$ is a non-empty Zariski-open subset of $V$ and hence Zariski-dense. 

Now, for all $\sigma \in U\cap U'$ we have:
\[
\dim V = \dim T_\sigma V = \dim  T_\sigma R_n(\pi_1M)_{red} \leq \dim T_\sigma R_n(\pi_1M) = \dim Z^1(\pi_1M;\Ad\sigma)\,.
\]

Next we will show that the hypothesis  $\dim V <n^2 + n -2-k$ implies that $\rho$ is non-reduced, and this contradicts the assumptions.
For all $\sigma \in U\cap U'$, we had:
\[
\dim V =  \dim T_\sigma R_n(\pi_1M)_{red} < n^2 -1-k + n-1\,.
\]
Notice that $\dim B^1(M;\Ad\sigma) =n^2-1-k$ and by {\it(\ref{red1})}  $\dim H^1(\pi_1 M;\Ad\rho)\geq n-1$. Hence
\[ 
 \dim T_\sigma R_n(\pi_1M)_{red} < n^2 -1-k+ n-1 \leq  \dim Z^1(\pi_1 M;\Ad\sigma) = \dim T_\sigma R_n(\pi_1M)\,.
 \]
 It follows that all representations $\sigma \in U\cap U'$ are non-reduced points of the representation scheme $R_n(\pi_1M)$ since
 $\dim T_\sigma R_n(\pi_1M)_{red} < \dim T_\sigma R_n(\pi_1M)$. On the other hand the non-reduced representations form a 
 Zariski-closed subset of $R_n(\pi_1M)$ (see \cite[Lemma~8]{HP23}). Therefore, we had  $\rho\in V = \overline{U\cap U'}$ is non-reduced. This is the desired contradiction.
\end{proof}

\subsection{The Alexander module}
\label{Alexander_module}
In this subsection we recall the main results from \cite{BenAH15} concerning the Alexander module and certain cohomological results that we will use in what follows (see also \cite{BenAbdelghani-Lines2002} and \cite{BenAbdelghani2000}).

Given a knot $K$ in a three-dimensional integer homology sphere $M^3$, we denote by $X=\overline{M^3\backslash V(K)}$ its exterior. Moreover, we let $\Gamma=\pi_1(X)$ denote the fundamental group and let  $h\co\Gamma\to\Z$ denote the abelianization morphism, so that $h(\gamma)=\mathrm{lk}(\gamma,K)$ is the linking number in $M^3$ between any loop realizing $\gamma\in\Gamma$ and $K$. 
There is a short exact splitting sequence:
\[ 1\to\Gamma'\to\Gamma\to Z\to 1\]
where $\Gamma'=[\Gamma,\Gamma]$ denotes the commutator subgroup of $\Gamma$, and $Z = \langle t\mid -\rangle$ is a free cyclic group generated by a distinguished generator $t$ represented by a meridian of $K$. The surjection $\Gamma\to Z$  is given
by $\gamma\mapsto t^{h(\gamma)}$. Hence $\Gamma$ can be written as a  semi-direct product $\Gamma = \Gamma' \rtimes Z$.
Note that $\Gamma'$ is the fundamental group of the infinite cyclic covering $X_\infty$ of $X$. So $H_1(X_\infty;\Z)\cong\Gamma'/\Gamma''$.
Moreover, the previous short exact splitting sequence induces the sequence:
\[ 1\to\Gamma'/\Gamma''\to\Gamma/\Gamma''\to Z \to 1\;.\]
Hence $\Gamma/\Gamma''$ can be written as $\Gamma/\Gamma''= \Gamma'/\Gamma'' \rtimes Z$.
 
As $\Lambda:=\C[t^{\pm1}]$ is a principal ideal domain, the complex version 
$\Gamma'/\Gamma''\otimes\C\cong H_1(\Gamma;\Lambda)$ of the Alexander module decomposes into a direct sum of cyclic modules of the form
$\Lambda/(t-\alpha)^k$, $\alpha\in\C^*\setminus\{1\}$ i.e. there exist
$\alpha_1,\ldots,\alpha_s\in\C^*\setminus\{1\}$ such that:

\[
H_1(\Gamma; \Lambda) \cong \tau_{\alpha_1}\oplus\cdots\oplus\tau_{\alpha_s}\quad\text{ where} \quad
\tau_{\alpha_j}=\bigoplus_{i_j=1}^{n_{\alpha_j}}\Lambda/(t-\alpha_j)^{r_{i_j}}
\]
denotes the $(t-\alpha_j)$-torsion of $H_1(\Gamma; \Lambda)$ and $n_j = \dim H_1(\Gamma;\Lambda/(t-\alpha_j))$.

A generator of the order ideal of $H_1(\Gamma;\Lambda)$ is called the \emph{Alexander polynomial} $\Delta_{K} \in \Lambda$ of $K$ i.e.
$\Delta_{K}$ is the product: 
\[
\Delta_{K}(t)=\prod_{j=1}^s \prod_{i_j=1}^{n_{\alpha_j}} (t-\alpha_j)^{r_{j_i}}\,.
\]
Notice that the Alexander polynomial is  symmetric and is well defined up to multiplication by a unit
$c\,t^k$ of
$\Lambda$, $c\in\C^*$, $k\in\Z$. It is well known that $\Delta_K$ can be normalized such that $\Delta_K$ has only integer coefficients, and 
$\Delta_K(1) = 1$. Hence the $(t-1)$-torsion of the Alexander module is trivial.
 
If $\alpha$ is a simple root of $\Delta_K$ then the $(t-\alpha)$-torsion of the Alexander module is of the form
$\C_\alpha = \Lambda/(t-\alpha)$. 
It follows from  Blanchfield-duality \cite[Chapter~7]{Gordon} that in this case the $(t-\alpha^{-1})$-torsion of the Alexander module is 
of the same form $\C_{\alpha^{-1}} = \Lambda/(t-\alpha^{-1})$, thus $\alpha^{-1}$ is also a simple root of
 $\Delta_K$.

For completeness we will state the next result which shows that
the cohomology groups $H^*(\Gamma; \Lambda/(t-\alpha)^k)$ are determined by the Alexander module
$H_1(\Gamma; \Lambda)$. Recall that the action of $\Gamma$ on
$\Lambda/(t-\alpha)^k$ is induced by $\gamma\cdot p(t) = t^{h(\gamma)} p(t)$.
Denote by $\pi_2$ the composition:
\[
\pi_2\co\Gamma' \to\Gamma'/\Gamma''\to\Gamma'/\Gamma'' \otimes\C\cong H_1(\Gamma;\Lambda)\;.
\]

\begin{prop}\label{prop:H1GammaC} Let $K\subset M^3$ be a knot in an integer homology sphere and $\Gamma = \Gamma'\rtimes Z$ its group.
Let $\alpha\in\C^*$  and let
$\tau_{\alpha} = \bigoplus_{i=1}^{n_\alpha} \Lambda \big/ (t-\alpha)^{r_i}$ denote
the $(t-\alpha)$-torsion of the Alexander module $H_1(\Gamma;\Lambda)$ where $n_\alpha = \dim H_1(\Gamma;\C_\alpha)$.
\begin{enumerate}
\item If $\alpha=1$ then $\C = \C_1$ is the trivial $\Gamma$-module and 
\[
H^q(\Gamma; \C)\cong
\begin{cases}
\C & \text{for $q=0,1$}, \\
0   & \text{for $q= 2$.}
\end{cases}
\]
\item For $\alpha\neq1$ with $\Delta_K(\alpha)\neq 0$ we have:
\[
H^q(\Gamma; \C_\alpha)\cong0 \quad\text{ for $q\geq 0$.} 
\]
\item 
If $\alpha\in\C^*$ is a  root of the Alexander polynomial $\Delta_K$ then the morphism

 \begin{align*}
 \psi\co 
 H^1(\Gamma'\rtimes Z;{\C}_{\alpha})&\to
 \operatorname{Hom}_\Lambda(\Gamma'/\Gamma''\otimes\C,\C_\alpha)  \\
 [V]&\mapsto \tilde{V}; \ \tilde{V}(\pi_2(x))=V(x,1)\ \forall x\in\Gamma'
\end{align*}
is an isomorphism.

\item For $\alpha$ a simple root of $\Delta_K$  we have:
 \[
H^q(\Gamma; \C_\alpha)\cong
\begin{cases}
0 & \text{for $q=0$}, \\
\C   & \text{for $q=1,2$.}
\end{cases}
\]
In particular, as
$H_1(\Gamma;\Lambda)\cong\Gamma'/\Gamma''\otimes\C=\Lambda/(t-\alpha)\oplus\tau$ where $\tau$ is a torsion $\Lambda$-module with no $(t-\alpha)$-torsion, $H^1(\Gamma;{\C}_{\alpha})$ is generated by the non-principal derivation $u_\alpha$ uniquely defined by the condition 
$\tilde u_\alpha(e_\alpha)=1$ and 
 $\tilde u_\alpha(y)  =0\ \forall y\in\tau$
where $(e_\alpha)$ is a basis of $\Lambda/(t-\alpha)$.
\end{enumerate}
\end{prop}
\begin{proof} For the proof see \cite [Proposition~2.1 and Proof of Theorem~3.1]{BenAbdelghani2000}.

\end{proof}

The next lemma forms the basis for our subsequent results and is implicitly contained in \cite [Proposition~3.5]{BenAH15}  
(see also \cite[Theorem~3.1]{BenAbdelghani2000}, \cite[Theorem~3.2]{BenAbdelghani-Lines2002}).

\begin{lemma}\label{lem:cupproduct}
Let $K\subset M^3$ be a knot in an integer homology sphere and $\Gamma$ its group.
Let $\alpha\in\C^*$ be a simple root of the Alexander polynomial $\Delta_K$ and 
$h\co \Gamma \to {\Z}$ the canonical epimorphism. The cup product:
\[
h \smallsmile\cdot\co H^1(\Gamma; {\C}_{\alpha^{\pm1}})\to H^2(\Gamma; {\C}_{\alpha^{\pm1}})
\]
is an isomorphism. In other words, if $u^\pm$ is a generator of $H^1(\Gamma;{\C}_{\alpha^{\pm1}})$ then
 $H^2(\Gamma;{\C}_{\alpha^{\pm1}})$ is generated by $h\smallsmile u^\pm$.
\end{lemma}
\begin{remark}\label{rq:cup}
Observe that by Equation~\eqref{eq:cupbil} the cocycles  $h\smallsmile u^\pm$ and $-u^\pm\smallsmile h$ are cohomologous. 

\end{remark}

\section{Proof of the first part of Theorem~\ref{thm:main1}}
\label{sec:main}

Our strategy consists in introducing a certain solvable representation $\rho^{(n)}_D\in R_n(\Gamma)$ whose orbit contains in its closure the diagonal representation $\rho_D$. As a consequence, the representations $\rho_D$ and $\rho^{(n)}_D$ are on the same component (see Subsection 3.1).
Using a theorem of \cite {Heusener-Medjerab2014} based on cohomological calculations, we prove that $\rho^{(n)}_D$ is a smooth point of $R_n(\Gamma)$ which is contained in a unique irreducible component $R^{(n)}_D$ of dimension $n^2+n-2$ (see Subsection 3.2). We finish the section by proving that $R^{(n)}_D$ contains irreducible representations by making use of Burnsides matrix lemma (see Subsection 3.3).
It is to mention that our strategy reproduces in a certain manner the one used by \cite  {Heusener-Porti-Suarez2001} for the study of the local structure of the representation space $R_2(\Gamma)$ in the neighborhood of an abelian representation which corresponds to a simple root of the Alexander polynomial of the knot. It is based on the fact that  $\rho^{(n)}_D$ is less singular than $\rho_D$ in the sense that $\dim O(\rho^{(n)}_D)=n^2-1$ whereas 
$\dim O(\rho_D)=n^2-n$ (see Proposition~\ref{prop:upper-triang-cohom}).

\subsection{The triangular representation}
\label{sec:SolvableRepresentation}

Let $D = D(\lambda_1,\ldots,\lambda_n)\in \SLn$ be a diagonal matrix with $\lambda_i\neq\lambda_j$ for
$i\neq j$. The aim of this section is to prove the existence of a certain upper triangular representation
$\rho^{(n)}_D\co\Gamma\to \SLn$. We will show that $\rho^{(n)}_D\in R_n(\Gamma)$ is a smooth point (see Theorem~\ref{thm:upper-triang-deform}), whereas the diagonal representation $\rho_D=D^h$ is not (see Remark~\ref{rem:components}).

\paragraph{Assumption.}
Throughout this section we will assume that
\[\alpha_i := \lambda_i/\lambda_{i+1},\ i=1,\ldots,n-1,\
\text{ are simple roots of the Alexander polynomial}\ \Delta_K\]
 and that 
\[\Delta_K(\lambda_i/\lambda_j)\neq0\ \text{if}\ |i-j|\geq 2.\]

\begin{lemma}\label{lem:upper-triang-exist}
If $u_i^+\in Z^1(\Gamma;\C_{\alpha_i})$, $i=1,\ldots,n-1$, are derivations then there exist  upper triangular representations $\rho^{(n)}_D\co\Gamma\to \SLn$ of the form:
\begin{equation}\label{eq:SolvableRep}
\rho^{(n)}_D(\gamma) =
\begin{pmatrix}
1 & u^+_1(\gamma) & * & \dots & * \\
0 & 1 & u^+_2(\gamma) & \ddots &\vdots \\
\vdots & \ddots & \ddots & \ddots &* \\
\vdots &  & \ddots & 1 & u^+_{n-1}(\gamma)\\
0 & \dots & \dots  & 0 & 1
\end{pmatrix}
D(\lambda_1,\ldots,\lambda_n)^{h(\gamma)}\;.
\end{equation}
Moreover,  the closure of the orbit of $\rho^{(n)}_D$ contains the diagonal representation $\rho_D$.

\end{lemma}

\begin{proof}
It is easy to verify that a map given by

\[
\rho(\gamma) =
\begin{pmatrix}
1 & z_{12}(\gamma) & * & \dots & z_{1n}(\gamma) \\
0 & 1 & z_{23}(\gamma) & \ddots &\vdots \\
\vdots & \ddots & \ddots & \ddots &* \\
\vdots &  & \ddots & 1 & z_{n-1,n}(\gamma)\\
0 & \dots & \dots  & 0 & 1
\end{pmatrix}
D(\lambda_1,\ldots,\lambda_n)^{h(\gamma)}
\]
is a representation if and only if for all $\ 1\leq i< j\leq n$ and all $\gamma_1,\gamma_2\in\Gamma$:
\begin{align*}
z_{ij} (\gamma_1\gamma_2) &= \sum_{k=i}^{j}z_{ik}(\gamma_1) \big(\frac{\lambda_k}{\lambda_j}\big)^{h(\gamma_1)} z_{kj}(\gamma_2)\\
&=  \big(\frac{\lambda_i}{\lambda_j}\big)^{h(\gamma_1)} z_{ij}(\gamma_2) +
\sum_{k=i+1}^{j-1}(z_{ik}\smallcup z_{kj})(\gamma_1,\gamma_2) + z_{ij}(\gamma_1)
\end{align*}
holds. Here we have used the convention that $z_{ii}(\gamma) =1$ for $1\leq i\leq n$ and $\gamma\in\Gamma$. 
\\
This equation is equivalent to
\begin{equation}\label{eq:equation_cup}
\delta z_{ij} + \sum_{i<k<j}z_{ik}\smallcup z_{kj} =0 \quad\text{ for all $1\leq i< j\leq n\,.$}
\end{equation}
It implies that $z_{i,i+1}\in Z^1(\Gamma;\C_{\alpha_i})$ if $1\leq i\leq n-1$ and 
\[
-\delta z_{ij} = 
\sum_{i<k<j}z_{ik}\smallcup z_{kj}\in B^2(\Gamma;\C_{\lambda_i/\lambda_{j}})
\quad \text{ if $j-i\geq2$.}
\]
By Proposition~\ref{prop:H1GammaC}, since $\lambda_i/\lambda_{j}$ is not a root of 
$\Delta_K$ for $|i-j|\geq2$, we have $H^2(\Gamma;\C_{\lambda_i/\lambda_j})=0$ and
 the cocycle $\sum_{i<k<j}z_{ik}\smallcup z_{kj}$ is a coboundary. Notice  also that by Proposition~\ref{prop:H1GammaC}, there exist non principal derivations $z\in Z^1(\Gamma;\C_{\alpha})$ if and only if $\alpha$ is  a root of $\Delta_K$.

For $t\neq 0$, let $C_t$ be the diagonal matrix $C_t:=\mathrm{diag}(t^{n-1},t^{n-2}, \ldots,1)$. Then it is easy to see that $\lim_{t \to0}C_t\rho^{(n)}_DC_t^{-1}=\rho_D$. This proves that 
$\rho_D\in\overline{ O(\rho^{(n)}_D)}$.
\end{proof}

\begin{remark}

\begin{enumerate}
\item The representation $\rho^{(n)}_D$ is non-abelian if  there exists $i$, $1\leq i\leq n-1$, such that 
$z_{i,i+1}= u_{i}^+\in Z^1(\Gamma;\C_{\alpha_i})$ is a non principal derivation. Indeed, by Lemma 2.1 of \cite{Heusener-Porti-Suarez2001} we can find a set of generators 
$\gamma_1, \ldots, \gamma_k$ of $\Gamma$ such that $h(\gamma_l)=1$ for $l=1,\ldots,k$. 
As $u_{i}^+$ is non-principal, there exist $\gamma_l$ and $\gamma_{l'}$ such that $u_{i}^+(\gamma_l)\neq u_{i}^+(\gamma_{l'})$. 
It follows that $\rho^{(n)}_D(\gamma_l\gamma_{l'})\neq \rho^{(n)}_D(\gamma_{l'}\gamma_l)$  since
\[
 u_{i}^+(\gamma_l\gamma_{l'}) =  u_{i}^+(\gamma_{l'}\gamma_{l}) \quad\Leftrightarrow\quad 
 \big(u_{i}^+(\gamma_l) - u_{i}^+(\gamma_{l'}) \big) (1-\alpha_i) = 0\,. 
\]

\item Once the derivations $u_i^+,\ i=1,\ldots,n-1$, are fixed, two such representations $\rho^{(n)}_D$ are conjugated. This claim is a consequence of the fact that $H^2(\Gamma;\C_{\lambda_i/\lambda_j})=0$ for $|i-j|\geq2$, and the fact that adding a coboundary $b(\gamma) = ((\lambda_i/\lambda_j)^{h(\gamma)}-1)a$ to $z_{ij}$, $j-i\geq2$, can be realized by conjugation by the matrix $I_n-aE_{i}^j$. Here, $(E_i^j)_{1\leq i,j\leq n}$ is the canonical basis of 
$\gln$.

\end{enumerate}

\end{remark}

The aim of this section is to prove the next theorem.
\begin{thm}\label{thm:upper-triang-deform}
Suppose that  $\alpha_i=\lambda_i/\lambda_{i+1}$, $i=1,\ldots,n-1$, are simple roots of the Alexander polynomial $\Delta_K$, and that $\Delta_K(\lambda_i/\lambda_j)\neq0$ if $|i-j|\geq 2$.
Then for given non-principal derivations $u_i^+\in Z^1(\Gamma;\C_{\alpha_i})$
the representation $\rho^{(n)}_D\co\Gamma\to \SLn$ given by
\eqref{eq:SolvableRep} is a scheme smooth representation of $R_n(\Gamma)$.

More precisely, $\rho^{(n)}_D$ is contained in a unique irreducible component $R^{(n)}_D \subset R_n(\Gamma)$ of dimension $n^2+n-2$ and the component $R^{(n)}_D$ contains irreducible representations.
\end{thm}

\begin{remark}\label{rem:components}
\begin{enumerate}
    \item 
The component $R^{(n)}_D \subset R_n(\Gamma)$ depends only on $\lambda_1,\ldots,\lambda_n$
 and not on the cocycles $u_i^+$.
This follows since the $u^+_i$ are non-principal and
$H^1(\Gamma;\C_{\alpha_i})\cong \C$. Adding a coboundary $b(\gamma) = (\alpha_i^{h(\gamma)}-1)a$ to $u_i^+$ and multiplying $u_i^+$ by a non-zero complex number $c\in\C^*$ can be realized by conjugation:
\[
\begin{pmatrix}
1 & a\\
0 &1
\end{pmatrix}
\begin{pmatrix}
1 & u_i^+(\gamma)\\
0 &1
\end{pmatrix}
\begin{pmatrix}
\lambda_i^{h(\gamma)} & 0\\
0 & \lambda_{i+1}^{h(\gamma)}
\end{pmatrix}
\begin{pmatrix}
1 & -a\\
0 &1
\end{pmatrix}
= 
\begin{pmatrix}
1 & u_i^+(\gamma)+ (\alpha_i^{h(\gamma)}  - 1)a\\
0 &1
\end{pmatrix}
\begin{pmatrix}
\lambda_i^{h(\gamma)} & 0\\
0 & \lambda_{i+1}^{h(\gamma)}
\end{pmatrix}
\]
and
\[
\begin{pmatrix}
c^{1/2} & 0\\
0 &c^{-1/2}
\end{pmatrix}
\begin{pmatrix}
1 & u_i^+(\gamma)\\
0 &1
\end{pmatrix}
\begin{pmatrix}
\lambda_i & 0\\
0 & \lambda_{i+1}
\end{pmatrix}
\begin{pmatrix}
c^{-1/2} & 0\\
0 &c^{1/2}
\end{pmatrix}
=
\begin{pmatrix}
1 & c\,u_i^+(\gamma)\\
0 &1
\end{pmatrix}
\begin{pmatrix}
\lambda_i & 0\\
0 & \lambda_{i+1}
\end{pmatrix}\,.
\]
\item By Lemma~\ref{lem:upper-triang-exist},
  the diagonal representation $\rho_D$ is contained in the closure of the orbit of $\rho^{(n)}_D$, thus $\rho_D$ is contained in the irreducible component $R^{(n)}_D$. On the other hand, as an abelian representation, $\rho_D$ is contained in the  irreducible component $R_n(\Z)$. Consequently, $\rho_D$ is not a smooth point of $R_n(\Gamma)$.
\end{enumerate}
\end{remark}

For the proof of Theorem~\ref{thm:upper-triang-deform} we will make use of the following result
(see \cite{BenAbdelghani-Heusener-Jebali2010,Heusener-Porti-Suarez2001,Heusener-Porti2005,Heusener-Medjerab2014, Heusener2016}).

\begin{prop}\label{prop:smoothpoint}

Let $M$ be a connected and orientable $3$-manifold with  toroidal boundary $\partial M \cong S^1\times S^1$ and let
$\rho\co\pi_1(M)\to \SLn$ be a representation.

If
$\dim H^1(\pi_1(M); \Ad\rho) =n-1$ then $\rho$   is a smooth point of the
$\SLn$-representation variety $R_n(\pi_1(M))$. More precisely,
$\rho$ is contained in a unique component of dimension
$n^2+n-2- \dim H^0(\pi_1(M);\Ad\rho)$.
\end{prop}

The knot complement $X$ has a single torus boundary component. To prove the first part of Theorem~\ref{thm:upper-triang-deform}, we will prove that
 $H^0(\Gamma;\Ad\rho^{(n)}_D)=0$ and
$\dim H^1(\Gamma; {\Ad\rho^{(n)}_D})=n-1$.
This calculation will occupy the following subsection.
The proof of the existence of irreducible representations in the component $R^{(n)}_D$ is given in Subsection~\ref{sec:irred}.

\subsection{Cohomolgical calculations}

Hereafter, we will fix non-principal derivations
$u_i^{\pm}\co\Gamma\to\C_{{\alpha_i}^{\pm}}$, $1\leq i\leq n-1$, and we let
$\rho^{(n)}_D$ be a representation as given by \eqref{eq:SolvableRep}. Moreover, we will consider
 $\sln$ (respectively $\gln$) as a $\Gamma$-module via $\Ad\rho^{(n)}_D$ and we will write simply
write $\sln$ (respectively $\gln$) for for these $\Gamma$-modules if no confusion can occur.
Notice that we have the following isomorphism of $\Gamma$-modules:
$
\gln \cong \sln \oplus \C\, I_n
$
where $\Gamma$ acts trivially on the center ${\C}I_n$ of $\gln$. It follows that:
\begin{equation}\label{eq:split}
H^*(\Gamma; \gln)\cong H^*(\Gamma; \sln)\oplus H^*(\Gamma; \C)\,.
\end{equation}

To compute the cohomology groups $H^*(\Gamma ; \gln)$
we will  use an adequate filtration of the $\Gamma$-module $\gln$.

Let $(E_1,\ldots,E_n)$ denote the canonical basis of the space of  column vectors.
Hence the matrices $E_i^j := E_i \, E_j^\intercal$, $1\leq i, j \leq n$, form the canonical basis of $\gln$.
We let $p_i^j\co \gln\to\C$ denote the corresponding linear coordinate functions.
For a given family $(X_i)_{i\in I}$, $X_i\in \gln$, we let $\langle X_i \,|\, i\in I \rangle\subset \gln$
denote the subspace of $\gln$ generated by the family.
For  a given $k\in\Z$, $1-n\leq k\leq n$, we define:
\begin{equation}\label{eq:Ck}
C_k= \langle E_i^j\in\mathfrak{gl}(n,\C)\ |\ j-i\geq k\rangle\,.
\end{equation}

For all $\gamma\in\Gamma$  we have:
\begin{equation}\label{eq:action}
\gamma\cdot E_i^j = (\rho^{(n)}_D(\gamma) E_i ) (E_j^\intercal \rho^{(n)}_D(\gamma^{-1}) )\,.
\end{equation}
More precisely, we obtain for $1\leq i,j\leq n$,
\[
\rho^{(n)}_D(\gamma) E_i =  
    \lambda_i^{h(\gamma)} E_i + \lambda_i^{h(\gamma)}u_{i-1}^+(\gamma) E_{i-1} +  \cdots 
\]
and
\begin{align*}
E_j^\intercal \rho^{(n)}_D(\gamma^{-1}) &=  
   \lambda_j^{-h(\gamma)} E_j^\intercal  + \lambda_{j+1}^{-h(\gamma)} u_j^+(\gamma^{-1}) E_{j+1}^\intercal+ \cdots \\
   &=    \lambda_j^{-h(\gamma)} E_j^\intercal  - \lambda_{j+1}^{-h(\gamma)} \big(\frac{\lambda_{j}}{\lambda_{j+1}}\big)^{-h(\gamma)} u_j^+(\gamma) E_{j+1}^\intercal+ \cdots\;.
\end{align*}
Here we have used the fact that, as $u^+_j$ is a derivation, $u_j^+(\gamma^{-1}) = - (\lambda_j/\lambda_{j+1})^{-h(\gamma)} u_j^+(\gamma)$.

Hence for $E_i^j \in C_k$ we have:
\begin{equation}\label{eq:action+}
\gamma\cdot E_i^j =  \big(\frac{\lambda_i}{\lambda_j}\big)^{h(\gamma)} E_i^j 
      + \big(\frac{\lambda_i}{\lambda_j}\big)^{h(\gamma)} u^+_{i-1}(\gamma)\,E_{i-1}^j 
      - \big(\frac{\lambda_i}{\lambda_{j}}\big)^{h(\gamma)} u_j^+(\gamma)\,E_i^{j+1} 
      +  c_{k+2}
\end{equation}
for a certain element $c_{k+2}\in C_{k+2}$.

It follows from equation~\eqref{eq:action+} that $C_k$ is a $\Gamma$-module.
Hence we obtain the descending sequence of $\Gamma$-modules:
$$\mathfrak{gl}(n,\C)=C_{1-n}\supset C_{2-n}\ldots\supset C_0\supset\ldots\supset C_{n-1}\supset C_n=\{0\}.$$

Notice that $C_0$ is the Borel subalgebra and that the coordinate function  
$p_1^n\co C_{n-1}=\left\langle E_1^n \right\rangle\to\C_{\lambda_1/\lambda_n}$
induces an isomorphism.

A basis of the vector space $C_k\big/C_{k+1}$ is given by:
\[
\begin{cases}
(E_1^{k+1} + C_{k+1}, E_2^{k+2}+C_{k+1} ,\ldots, E^n_{n-k}+C_{k+1}) & \text{if $0\leq k\leq n-1$,}\\
(E^1_{|k|+1} + C_{k+1}, E^2_{|k|+2} + C_{k+1}, \ldots, E_n^{n-|k|}+C_{k+1}) &\text{if $1-n\leq k\leq 0$.}
\end{cases}
\]
Hence we have the following isomorphisms which are induced by the projections:
\begin{align*}
p_1^{k+1} \oplus p_2^{k+2} \oplus\cdots\oplus p^n_{n-k}\co&
C_k\big/C_{k+1}\xrightarrow{\cong} \oplus_{i=1}^{n-k} \C_{\lambda_i/\lambda_{k+i}}
&\text{ if $0\leq k\leq n-1$,}\\
p^1_{|k|+1} \oplus p^2_{|k|+2} \oplus\cdots\oplus p_n^{n-|k|}\co &
C_k\big/C_{k+1}\xrightarrow{\cong} \oplus_{i=1}^{n-|k|} \C_{\lambda_{i+|k|}/\lambda_{i}}
&\text{ if $1-n\leq k<0\,.$}
\end{align*}

As $\lambda_i/\lambda_j$, $|i-j|\geq2$, is not a root of the Alexander polynomial $\Delta_K$
 and 
$\lambda_i/\lambda_{i+1}$, $1\leq i\leq n-1$, is a simple root of $\Delta_K$, we obtain from Proposition~\ref{prop:H1GammaC}:
\begin{align}
H^0(\Gamma; C_k\big/C_{k+1}) &\cong
\begin{cases}
0 &\text{ if $k\neq 0$},\\
\C^n &\text{ if $k=0$},
\end{cases}\notag\\
H^1(\Gamma; C_k\big/C_{k+1}) &\cong
\begin{cases}
0 &\text{ if $k\neq-1,0,1$},\\
\C^{n} &\text{ if $k=0$},\\
\C^{n-1} &\text{ if $k=\pm1$},
\end{cases}\label{eq:H^*}\\
H^2(\Gamma; C_k\big/C_{k+1}) &\cong
\begin{cases}
0 &\text{ if $k\neq\pm1$},\\
\C^{n-1} &\text{ if $k= \pm1\,.$}
\end{cases}\notag
\end{align}

In what follows we will consider the following short exact sequence:
\begin{equation}\label{suite:ck}
0\to C_{k+1}\stackrel{i_k}{\rightarrowtail} C_k\stackrel{pr_k}{\twoheadrightarrow }C_k\big/C_{k+1}\to 0\,.
\end{equation}
For $|k|\leq n-1$, the sequence (\ref{suite:ck}) induces a long exact sequence in cohomology:
\begin{multline} \label{eq:seqC_k}
0\to H^0(\Gamma;C_{k+1})\to H^0(\Gamma;C_k)  \to
H^0(\Gamma;C_k/C_{k+1})\xrightarrow{\beta^0_k}\\
H^1(\Gamma;C_{k+1}) \to  H^1(\Gamma;C_k)
\to H^1(\Gamma;C_{k}/C_{k+1})\xrightarrow{\beta^1_k} \\H^2(\Gamma;C_{k+1}) \to
  H^2(\Gamma;C_k)\to   H^2(\Gamma;C_k/C_{k+1})\to \cdots
\end{multline}
where we let
$\beta_k^*\co H^*\left(\Gamma;C_k\big/C_{k+1}\right)\to H^{*+1}(\Gamma;C_{k+1})$ denote
the connecting homomorphism.

\begin{prop}\label{prop:upper-triang-cohom}
Suppose that  $\alpha_i=\lambda_i/\lambda_{i+1}$, $i=1,\ldots,n-1$, are simple roots of the Alexander polynomial $\Delta_K$, and that $\Delta_K(\lambda_i/\lambda_j)\neq0$ if $|i-j|\geq 2$.
Then for given non-principal derivations $u_i^+\in Z^1(\Gamma;\C_{\alpha_i})$,
the representation $\rho^{(n)}_D\co\Gamma\to \SLn$ given by
\eqref{eq:SolvableRep} verifies:
\[
H^0(\Gamma;\sln)=\{0\}\quad\text{ and }\quad
\dim H^1(\Gamma;\sln)=n-1\,.
\]
\end{prop}

\begin{proof}
It follows from  \eqref{eq:H^*} that $H^2(\Gamma; C_k\big/C_{k+1})=0$ for $|k|\geq2$; hence
from the long exact sequence \eqref{eq:seqC_k} we obtain isomorphisms
$H^*(\Gamma;C_{k+1})\xrightarrow{\cong} H^*(\Gamma;C_k)$ for $|k|\geq2$.

From this and \eqref{eq:H^*} we obtain:
\[
0=H^*(\Gamma;C_{n-1})\xrightarrow{\cong} H^*(\Gamma;C_{n-2}) \xrightarrow{\cong} \cdots
\xrightarrow{\cong} H^*(\Gamma;C_{2})
\]
and
\[
H^*(\Gamma;C_{-1})\xrightarrow{\cong} H^*(\Gamma;C_{-2}) \xrightarrow{\cong} \cdots
\xrightarrow{\cong} H^*(\Gamma;C_{1-n}) =
H^*(\Gamma;\mathfrak{gl}(n,\C)) \,.
\]
Now, $H^*(\Gamma;C_2)=0$ and $C_2\subset C_{-1}$ imply that:
\[
H^*(\Gamma;C_{-1}) \xrightarrow{\cong} H^*(\Gamma;C_{-1}/C_2)\,.
\]
Therefore it will be sufficient to compute the dimension of $H^*(\Gamma;C_{-1}/C_2)$.

We have the following inclusions:
\[
C_{-1}\supset C_0\supset C_1 \supset C_2 \quad\Rightarrow\quad
C_{-1}/C_2 \supset C_0/C_2\supset C_1/C_2\,.
\]
Thus we obtain two exact sequences of $\Gamma$-modules:
\begin{align}
0\to C_1/C_2 \to C_0/C_2 \to C_0/C_1\to 0\,, \label{eq:SES1}\\
0\to C_0/C_2 \to C_{-1}/C_2 \to C_{-1}/C_0\to 0\,. \label{eq:SES2}
\end{align}
The short exact sequence \eqref{eq:SES1}
induces the following long sequence in cohomology:
\begin{multline*}
0\to H^0(\Gamma;C_0/C_2)  \to
H^0(\Gamma;C_0/C_1)\xrightarrow{\beta^0}
H^1(\Gamma;C_1/C_2)
\to  H^1(\Gamma;C_0/C_2)\\
\to H^1(\Gamma;C_0/C_1)\xrightarrow{\beta^1} H^2(\Gamma;C_1/C_2)
\to  H^2(\Gamma;C_0/C_2)\to 0
\end{multline*}
since by equation~\eqref{eq:H^*}, $H^0(\Gamma;C_1/C_2)=0$ and
$H^2(\Gamma;C_0/C_1)=0$.

In order to determine the connection operators
 $\beta^*$, $*=0,1$, we fix a basis $E_i^i +C_1$, $i=1,\ldots,n$,
 of  $C_0/C_1$. Then $\{E_i^i +C_1\}_{1\leq i\leq n}$ is a basis of $H^0(\Gamma;C_0/C_1)$,
 and the cocycles $U_j\co\Gamma\to C_1/C_2$, $j=1,\ldots,n-1$, given by
 $U_j(\gamma)= u_{j}^+(\gamma)E_{j}^{j+1} +C_2$ represent a basis of
$H^1(\Gamma;C_1/C_2)$.

Now we let $\delta^*$ denote the coboundary operator of the complex
$C^*(\Gamma; \gln)$. Then equation~\eqref{eq:action+} gives directly:
\[
\delta^0(E_i^i)(\gamma)= \gamma\cdot E_i^i - E_i^i = u_{i-1}^+(\gamma)\, E^{i}_{i-1} - u_{i}^+(\gamma)\, E^{i+1}_{i} + c_2
\]
for a certain element $c_2\in C_2$. Therefore, $\beta^0(E_i^i +C_1) = U_{i-1} - U_i +C_2$
where $U_0$ and $U_n$ are understood to be zero. Hence the matrix of $\beta^0$ with respect to the basis $\{E_i^i +C_2\}_{1\leq i\leq n}$ and $\{U_j\}_{1\leq j\leq n-1}$ of $H^0(\Gamma;C_0/C_1)$ and
$H^1(\Gamma;C_1/C_2)$ respectively is the $(n-1)\times n$ matrix:
\begin{equation}\label{eq:matrixM}
M= \begin{pmatrix}
-1 & 1 & 0 & \dots & 0\\
0 &  -1  & 1 & 0 & \vdots \\
\vdots & \ddots & \ddots & \ddots & 0 \\
0 &       \dots    & 0       & -1 & 1
\end{pmatrix}
\end{equation}
of rank $n-1$. Hence $\beta^0$ is surjective.

Similarly, $\{h\,E_i^i +C_2\}_{1\leq i\leq n}$ represents a basis of $H^1(\Gamma;C_0/C_1)$, and
\[
\delta^1(h\,E_i^i) (\gamma_1,\gamma_2)= (u_{i-1}^+\smallcup h) (\gamma_1,\gamma_2)\,E_{i-1}^i 
- (u_i^+\smallcup h)(\gamma_1,\gamma_2)\, E_i^{i+1} + c_2\,,
\]
where $u_0^+$ and $u_n^+$ are understood to be zero. By Lemma~\ref{lem:cupproduct} and Remark~\ref{rq:cup}, the 
2-cocycles $u_i^+ \smallcup h\in Z^2(\Gamma ; \C_{\alpha_i})$ represent generators of $H^2(\Gamma;\C_{\alpha_i})$.
Hence the matrix of $\beta^1$ with respect to the bases $\{h\,E_i^i +C_2\}_{1\leq i\leq n}$ and $\{(u_i^+\smallcup h)\, E_i^{i+1} +C_2\}_{1\leq i\leq n-1}$ of $H^1(\Gamma;C_0/C_1)$ and
$H^2(\Gamma;C_1/C_2)$ respectively is also the $(n-1)\times n$ matrix $M$ of equation \eqref{eq:matrixM}.
Hence $\beta^1$ is also surjective.

As a consequence $H^2(\Gamma;C_0/C_2)=0$
and we obtain the two short exact sequences:
\[
0\to H^0(\Gamma;C_0/C_2)\to H^0(\Gamma;C_0/C_1)\to
H^1(\Gamma;C_1/C_2)\to0\,,
\]
and
\[
0\to H^1(\Gamma;C_0/C_2)\to H^1(\Gamma;C_0/C_1)\to
H^2(\Gamma;C_1/C_2)\to0\,.
\]
We deduce from equation~\eqref{eq:H^*} that $\dim H^0(\Gamma;C_0/C_2)=\dim H^1(\Gamma;C_0/C_2)=1$.

Observe that $H^1(\Gamma;C_0/C_2)$ is generated by $\langle h\, I_n +C_2\rangle$, where $I_n$ is the identity matrix of order $n$.

Now consider the short exact sequence \eqref{eq:SES2}
which induces the following sequence in cohomology:
\begin{multline*}
0\to H^0(\Gamma;C_0/C_2)\to H^0(\Gamma;C_{-1}/C_2)\to H^0(\Gamma;C_{-1}/C_0)\to
H^1(\Gamma;C_0/C_2)\\
\to  H^1(\Gamma;C_{-1}/C_2)
\to H^1(\Gamma;C_{-1}/C_0)\to H^2(\Gamma;C_0/C_2) \\
\to  H^2(\Gamma;C_{-1}/C_2)\to H^2(\Gamma;C_{-1}/C_0)\to
 \cdots\,.
\end{multline*}
As $H^0(\Gamma;C_{-1}/C_0)=0$ and $H^2(\Gamma;C_0/C_2)=0$ we obtain the isomorphism:
\[
 H^0(\Gamma;C_0/C_2)\xrightarrow{\cong}H^0(\Gamma;C_{-1}/C_2)
\] 
and the short exact sequence
\begin{equation}\label{eq:traceless}
0\to H^1(\Gamma;C_0/C_2)
\to  H^1(\Gamma;C_{-1}/C_2)
\to H^1(\Gamma;C_{-1}/C_0)\to0\,.
\end{equation}
In conclusion
\[
\begin{cases}
\dim H^0(\Gamma;\gln)=\dim H^0(\Gamma;C_{-1}/C_2)=\dim  H^0(\Gamma;C_0/C_2) = 1\,,\\
\dim H^1(\Gamma;\gln)=\dim H^1(\Gamma;C_{-1}/C_2)=1 + (n-1) = n\,.
\end{cases}
\]
Now, equation \eqref{eq:split} gives $\dim H^0(\Gamma;\sln)=0$, and
$\dim H^1(\Gamma;\sln)=n-1$. 

\end{proof}

\begin{remark}\label{rem:trace-cohom}
The cohomology groups $ H^0(\Gamma;C_0/C_2)$ and $H^1(\Gamma;C_0/C_2)$ are generated by
$I_n +C_2$ and $h\,I_n+C_2$ respectively where $I_n$ denotes the identity matrix. This follows since $\Ker\beta^0$ and
$\Ker\beta^1$ are generated by $I_n +C_1$ and $h\,I_n+C_1$ respectively.
\end{remark}

We call a cochain $c\co\Gamma\to\gln$ \emph{traceless} if $\tr(c(\gamma))=0$ for all $\gamma\in\Gamma$. Notice that a cochain is traceless if and only if it takes values in $\sln\subset\gln$. Every principal derivation $b\co\Gamma\to\gln$ is traceless, and hence we can define traceless cohomology classes.

It is clear that each traceless cocycle in $Z^1(\Gamma;C_{-1})$ represents a cohomology class in
$H^1(\Gamma;\sln)$. Reciprocally, we have the following:
\begin{cor}\label{cor:H^1sln}
Each element of $H^1(\Gamma;\sln)$ can be represented by a traceless cocycle
in $Z^1(\Gamma;C_{-1})$. Moreover, two traceless cocycles in $Z^1(\Gamma;C_{-1})$ represent the same cohomology class in $H^1(\Gamma;\sln)$ if and only if their projections represent the same cohomology class in $H^1(\Gamma;C_{-1}/C_0)$. In other words, the composition map:
\[
H^1(\Gamma;\sln)\hookrightarrow H^1(\Gamma;\gln) \cong H^1(\Gamma;C_{-1})\to H^1(\Gamma;C_{-1}/C_0)
\]
is an isomorphism.
\end{cor}
\begin{proof}
The first part follows easily since a principal derivation with values in $\gln$ is traceless, and since
$H^1(\Gamma;C_{-1})\xrightarrow{\cong} H^1(\Gamma;\gln)$ is an isomorphism.
More precisely, traceless cohomology classes in $H^1(\Gamma;C_{-1})$ correspond exactly to
cohomology classes in $H^1(\Gamma;\sln)\subset H^1(\Gamma;\gln)$.
Also, each cochain $c\co\Gamma\to C_2$ is traceless, and therefore traceless cohomology classes in
$H^1(\Gamma;C_{-1}/C_2)$ are defined.
Next, notice that a traceless cohomology class in $H^1(\Gamma;C_{-1})$ projects onto a traceless cohomology class in $H^*(\Gamma;C_{-1}/C_2)$. We consider the sequence \eqref{eq:traceless} 
\[
0\to H^1(\Gamma;C_0/C_2)
\to  H^1(\Gamma;C_{-1}/C_2)
\to H^1(\Gamma;C_{-1}/C_0)\to0\,.
\]
If a cohomology class in
$H^1(\Gamma;C_{-1}/C_2)$ projects onto the trivial class in $H^1(\Gamma;C_{-1}/C_0)$ then it is in the image of
$H^1(\Gamma;C_{0}/C_2)\to H^1(\Gamma;C_{-1}/C_2)$. By Remark~\ref{rem:trace-cohom}
$\operatorname{Im}\big(H^1(\Gamma;C_{0}/C_2)\to H^1(\Gamma;C_{-1}/C_2)\big)$ is generated by $h\,I_n+C_2$. It follows that the class is trivial or non-traceless.
We obtain the following isomorphism:
\[
H^1(\Gamma;\sln) = H_0^1(\Gamma;\gln)
 \cong H_0^1(\Gamma;C_{-1})
\xrightarrow{\cong} H_0^1(\Gamma;C_{-1}/C_2) \xrightarrow{\cong} H^1(\Gamma;C_{-1}/C_0) \,.
\]
Here $H^1_0(.)$ denotes the traceless cohomology classes.
\end{proof}

Propositions~\ref{prop:smoothpoint} and \ref{prop:upper-triang-cohom} prove the first part of Theorem~\ref{thm:upper-triang-deform}. The last part of Theorem~\ref{thm:upper-triang-deform} will be proved in the next subsection.

\subsection{Irreducible $\mathrm{SL}(n,\C)$ representations}\label{sec:irred}

\begin{proof}[Proof of the last part of Theorem~\ref{thm:upper-triang-deform}]
In order to prove that there are irreducible deformations of $\rho^{(n)}_D$, we will make use of Burnsides matrix lemma stating that a representation $\rho\co\Gamma\to\GLn$ is irreducible if and only if
the image $\rho(\Gamma)$ generates the matrix algebra $M_n(\C)$.

Let $\Gamma=\langle S_1,\ldots,S_k\;|\;W_1,\ldots,W_{k-1}\rangle$ be a presentation of the  knot group $\Gamma$ such that $h(S_i)=1$ for all $1\leq i \leq k$. Such a presentation of deficency one exists always (see Section~2 in \cite{Heusener-Porti-Suarez2001}).

Recall that
$\rho^{(n)}_D\co\Gamma\to \SLn$ is a representation given by \eqref{eq:SolvableRep} and satisfying the conditions of Theorem~\ref{thm:upper-triang-deform}:
\[
\rho^{(n)}_D(\gamma) =
\begin{pmatrix}
1 & u_{1}^+(\gamma) & z_{13}^+(\gamma) & \dots & z_{1n}^+(\gamma) \\
0 & 1 & u_{2}^+(\gamma) &\ddots  &\vdots \\
\vdots & \ddots & \ddots & \ddots &z_{n-2,n}^+(\gamma) \\
\vdots &  &  \ddots& 1 & u_{n-1}^+(\gamma)\\
0 & \cdots & \cdots  & 0 & 1
\end{pmatrix}
D(\lambda_1,\ldots,\lambda_n)^{h(\gamma)}\,.
\]

Modulo conjugation
of the representation $\rho^{(n)}_D$, we can assume that $u_1^+(S_1)=\ldots=u_{n-1}^+(S_1)=0$ and that $z_{ij}^+(S_1)=0$
for $j-i\geq2$, $1\leq i<j\leq n$. This conjugation corresponds to
adding  principal derivations to the derivations $u_i^+$, $1\leq i\leq n-1$, and to the cochains $z_{ij}^+$;
 $j-i\geq2$. Thus we can suppose that $\rho^{(n)}_D(S_1)=D(\lambda_1,\ldots,\lambda_n)$.

Since the derivations $u_i^+$ are not principal there exists for each $i$, $1\leq i \leq n-1$, a generator $S_{j_i}$, $j_i\in\{2,...,k\}$ such that $u_i^+(S_{j_i})\neq 0$. Now, for each $i$, $1\leq i \leq n-1$, we chose a non-principal derivation $u_i^- \in Z^1(\Gamma; \C_{\lambda_{i+1}/\lambda_i})$.
 By Corollary~\ref{cor:H^1sln}, there exists a cocycle
 $V_{n-1}\in Z^1(\Gamma;\sln)$ of the form:
\[
V_{n-1}=   u_1^-E_2^1+u_2^-E_{3}^2+\ldots+u_{n-1}^-E_n^{n-1}+X,
\]
where $X\in
C^1\left(\Gamma;C_0\right)$. Up to adding  principal derivations to the derivations
 $u_i^-$ we may assume that
$u_i^-(S_1)=0,\ \forall\ 1\leq i\leq n-1$. By arguing as in the proof of Lemma 5.5 of \cite{BenAbdelghani-Heusener-Jebali2010}
it can be  proved that  $u_i^-(S_{j_i})\neq0$, $1\leq i\leq n-1$.

Let $\rho_t$ be a deformation of $\rho^{(n)}_D$ with leading term $V_{n-1}$:
\[
\rho_t = \big( I_n + t\, V_{n-1} + o(t) \big) \rho^{(n)}_D\,, \text{ where } \lim_{t\to 0} \frac{o(t)}{t} =0\,.
\]

As the set of diagonal matrices with non vanishing pairwise different entries is Zariski open, we may suppose that
for sufficiently small $t$, $\rho_t(S_1) = D(a_{1}(t),\ldots,a_{n}(t))$
is such that $a_{i}(t)\neq 0$, and $a_{i}(t)\neq a_{j}(t)$ for $i\neq j$.
Recall that the diagonal elements $\lambda_i$ are pairwise different. 
It follows that the matrix $E_{i}^i$ is a polynomial in $\rho_t(S_1)$. For if $\C[\rho_t(S_1)]$ denotes the algebra generated by $\rho_t(S_1)$ we have $\C[E_1^1,\ldots,E_n^n]\supset \C[\rho_t(S_1)]$,  $\dim \C[\rho_t(S_1)] =n$, and hence
$\C[E_1^1,\ldots,E_n^n] = \C[\rho_t(S_1)]$.

Let $\mathcal{A}_t$ denote the algebra generated by $\rho_t(\Gamma)$.
Since $\rho_t(S_1)\in \mathcal{A}_t$ we have that $E_{i}^{i}\in \mathcal{A}_t$ for all $1\leq i\leq n$.
Hence if $B=(b_{ij})\in \mathcal{A}_t$ then:
\[
 E_{i}^{i} B E_{j}^{j} = b_{ij} E_{i}^{j}\in \mathcal{A}_t\,.
\]

We consider
\[
 E_{i}^{i} \, \rho_t(S_{j_i}) \, E_{i+1}^{i+1} = \big(\lambda_{i+1} u_i^+(S_{j_i}) + o(1)\big) E_{i}^{i+1}\,.
\]
We have $\lambda_{i+1} u_i^+(S_{j_i})\neq 0$, thus for $t$ sufficently small we have $E_{i}^{i+1}\in\mathcal{A}_t$.

Similarly,
\[
 E_{i+1}^{i+1} \, \rho_t(S_{j_i}) \, E_{i}^{i} = \big(t\lambda_{i} u_i^-(S_{j_i}) +o(t) \big) E_{i+1}^{i}\;
\]
We have $\lambda_{i} u_i^-(S_{j_i})\neq 0$ and hence for sufficiently small $t\neq 0$ we have
$t (\lambda_{i} u_i^-(S_{j_i}) +o(t)/t)\neq0$. Therefore $E_{i+1}^{i}\in\mathcal{A}_t$.

Finally,  Lemma~\ref{lem:alg} implies that $\mathcal{A}_t=M_n(\C)$ and
$\rho_t$ is irreducible.
\end{proof}
\begin{lemma} \label{lem:alg}
The $2(n-1)$ matrices $\{ E_{i}^j \mid |i-j|=1\}$ generate the full matrix algebra $M_n(\C)$.
\end{lemma}
\begin{proof}
Let $\mathcal{A} := \C[E_{i}^j \mid |i-j|=1]$.
We will proof that all elementary matrices $E_{i}^j$,  $1\leq i\leq n$, $1\leq j\leq n$, are in  $\mathcal{A}$.

For $i=1$ a direct calculation gives:
\[
E_{1}^1 = E_{1}^2 E_{2}^{1}, \quad E_{1}^{3} = E_{1}^{2} E_{2}^{3}, \quad E_{1}^{4} = E_{1}^{3} E_{3}^{4},
\quad \ldots \quad ,
E_{1}^{n} = E_{1}^{n-1} E_{n-1}^{n}\,.
\]
Hence $E_{1}^{j}\in \mathcal{A}$ for $1\leq j\leq n$. 

Suppose now that $E_{i}^{j}\in \mathcal{A}$ for $1\leq j\leq n$ and let us prove that  $E_{i+1}^{j}\in \mathcal{A}$ for $1\leq j\leq n$.
Clearly $E_{i+1}^{i}, E_{i+1}^{i+2}, E_{i+1}^{i+1} = E_{i+1}^{i} E_{i}^{i+1} \in \mathcal{A}$. 
For $1\leq j\leq n$ and $j\notin\{i,i+1,i+2\}$, we have $E_{i+1}^{j} = E_{i+1}^{i} E_{i}^{j}\in \mathcal{A}$. 
This ends the proof.
\end{proof}

\section{The  diagonal representation}
\label{sec:diagonal}

The aim of this section is to prove that every vector in the tangent cone of $\rho_D$ is integrable (see Theorem~\ref{thm:main3}).
Throughout this section we will assume that $\lambda_i/\lambda_{i+1}$, $i=1,\ldots,n-1$, is a simple root of the Alexander polynomial $\Delta_K$, and that $\Delta_K(\lambda_i/\lambda_j)\neq0$ if $|i-j|\geq 2$.
Notice that, by the symmetry of $\Delta_K$, the quotients $\lambda_{i+1}/\lambda_{i}$
are also simple roots of the Alexander polynomial (see the paragraph preceeding Proposition~\ref{prop:H1GammaC}).

Let $D = D(\lambda_1,\ldots,\lambda_n)\in \SLn$ be the diagonal matrix with $\lambda_i\neq\lambda_j$ for
$i\neq j$ and let $\rho_D$ be the abelian representation given by $\rho_D(\gamma) = D^{h(\gamma)}$.

Recall that the quadratic cone at a representation $\rho\in R_n(\Gamma)$, introduced by W.~Goldman \cite{Goldman1984,Goldman1985}, is an obstruction to integrability:
\begin{equation}\label{eq:cone}
TQ_{\rho}\big(R_n(\Gamma)\big)=\{ U \in Z^1(\Gamma;{\Ad\rho}) \mid [U\smallcup U]\sim0\}
\subset T_{\rho} R_n(\Gamma)\cong Z^1(\Gamma;{\Ad\rho}) \,.
\end{equation}

We will calculate the quadratic cone at the abelian representation $\rho_D$. More precisely, first we will compute a basis of the Zariski tangent space $Z^1(\Gamma;{\Ad\rho_D})$ (see Subsection~\ref{sec:Tangent_space_diag}), then we will determine the ideal $I$ generated by the quadratic obstruction \eqref{eq:cone} (see Subsection~\ref{sec:quad_cone_diag}).
After that  we will show 
in Subsection~\ref{sec:ideal}
that the ideal $I$ is radical (see Lemma~\ref{lem:radical}), and that $I$  is in fact an intersection of prime ideals generated by linear forms (see Lemma~\ref{lem:prime_dec}).
This guarantees that the quadratic cone is reduced.
In Subsection~\ref{sec:cone}, we will describe the quadratic cone at $\rho_D$ in terms of the decomposition of the ideal $I$.
Finally, we will prove in Subsection~\ref{sec:tangent_cone} that every cocycle in the quadratic cone 
$TQ_{\rho_D}\big(R_n(\Gamma)\big)$ is integrable, and hence the quadratic cone coincides with the tangent cone.

\subsection{The Zariski tangent space at the diagonal representation}
\label{sec:Tangent_space_diag}

 The $\Gamma$-module $\sln$ (via $\Ad\rho_D$) decomposes as a direct sum of on-dimensional
$\Gamma$-modules. More precisely, we have

\begin{equation}\label{eq:decomp_sln}
\sln  \cong  \bigoplus_{i\neq j} \C_{\lambda_i/\lambda_j} \oplus \C^{n-1}\,.
\end{equation}

In what follows we fix the  non principal derivations:
\[
 h\in Z^1(\Gamma;\C),\quad
u_i^+\in Z^1(\Gamma;\C_{\lambda_i/\lambda_{i+1}}), \quad
u_i^-\in Z^1(\Gamma;\C_{\lambda_{i+1}/\lambda_{i}}),\quad i=1,\ldots,n-1\,.
\]
Notice that $Z^1(\Gamma;\C)= H^1(\Gamma;\C)$ since $\C$ is a trivial $\Gamma$-module.

Recall that $E_i^j= E_i \, E_j^\intercal$, $1\leq i, j \leq n$, form the canonical basis of $M_n(\C)$
and that  $E_i^j\,E_k^l = \delta_k^j\, E_i^l$. Hence
\begin{equation}\label{eq:bracked}
[E_i^j , E_k^l ] = \delta_k^j\, E_i^l - \delta_i^l\, E_k^j\,.
\end{equation}

In what follows we will consider the cohomology classes represented by the following cocycles:
\begin{equation}\label{eq:generators}
H_i = h (E_{i}^i - E_{i+1}^{i+1}),\quad
U_i^+ = u_i^+ E_{i}^{i+1}\text{ and }
U_i^- = u_i^- E_{i+1}^{i}\text{ where }\ i\in\{1,\ldots,n-1\}\,.
\end{equation}

\begin{lemma}\label{elm:basis-cohomology-abelian}
The $3n-3$ cohomology classes represented by 
$H_i$, \(U_i^+\) and \(U_i^-\), $i\in\{1,\ldots,n-1\}$,
form a basis of 
$H^1(\Gamma;{\Ad\rho_D})$.
\end{lemma}
\begin{proof}
Use Proposition~\ref{prop:H1GammaC} and equations \eqref{eq:generators}.
\end{proof}

Roughly speaking, this means that the cohomology
$H^1(\Gamma;{\Ad\rho_D})$ is concentrated around the diagonal, and
\begin{align}
\dim Z^1(\Gamma;{\Ad\rho_D}) &= n^2+2n-3,\notag\\
\dim B^1(\Gamma;{\Ad\rho_D}) &= n^2-n,\label{eq:dimZ1}\\
\dim H^1(\Gamma;{\Ad\rho_D}) &= 3(n-1)\notag\,.
\end{align}
As the homotopy type of the knot complement is $2$-dimensional, $H^3(\Gamma;{\Ad\rho_D})$ vanishes, and the Euler characteristic will be a multiple of that of $\partial M$ which is toroidal, and thus vanishes. Hence the alternating sum of the twisted Betti numbers also vanishes.

Using the fact that $\dim H^0(\Gamma;{\Ad\rho_D}) =(n-1)$ we obtain that:
\begin{equation}\label{eq:dim_sln}
\dim H^2(\Gamma;{\Ad\rho_D}) = 2(n-1).
\end{equation}

\begin{lemma}\label{lem:tgt-space-variety}
For the representation $\rho_D\in R_n(\Gamma)$ we have an isomorphism
\[
 T_{\rho_D} R_n(\Gamma)_\mathit{red}\xrightarrow{\cong} Z^1(\Gamma;{\Ad\rho_D})\,.
\]
\end{lemma}
\begin{proof}
Every principal derivation is integrable, and the derivations \eqref{eq:generators} are  seen to be integrable:
\begin{equation}\label{eq:path_metabel}
\rho_t^{(H_i)} (\gamma) = \exp\big( t H_i(\gamma) \big) \rho_D(\gamma),\quad\text{ and }\quad
\rho_t^{(U_i^\pm)} (\gamma) = \exp\big( t U_i^\pm(\gamma) \big) \rho_D(\gamma)\,.
\end{equation}
This follows since diagonal matrices commute, and $(E_i^{i+1})^2=(E_{i+1}^{i})^2=0$.
\end{proof}

\begin{remark}
The fact that the space of $1$-cocycles $Z^1(\Gamma;{\Ad\rho_D})$ is isomorphic to the Zariski tangent 
space $ T_{\rho_D} R_n(\Gamma)_\mathit{red}$ does not imply that the representation $\rho_D$ is scheme reduced. 
In order to prove this property for $\rho_D$, we will show in Theorem~\ref{thm:tangent-cone} that the tangent cone at $\rho_D$ is reduced and then apply Lemma~\ref{lem:scheme2}.
\end{remark}

\subsection{The cup product of cocycles at the diagonal representation}
\label{sec:quad_cone_diag}
In general, each cocycle $U \in Z^1(\Gamma;{\Ad\rho_D})$ has the form:
\begin{equation}\label{eq:cocycle}
U =\sum_{i=1}^{n-1} x_i U_i^+ + y_i U_i^- + z_i H_i + C
\end{equation}
where $C\in B^1(\Gamma;{\Ad\rho_D})$ is a principal derivation, and $x_i,y_i,z_i\in\C$.

\begin{lemma}\label{lem:cup_trivial}
Let $H_i$, and $U_i^\pm$ be the cocycles defined in Equation~\eqref{eq:generators}.Then
\[
[H_i\smallcup H_j]= 0,\quad [U_i^\pm\smallcup U_j^\pm]\sim 0, \text{ and } [U_i^\mp\smallcup U_j^\pm]\sim0\,.
\]
\end{lemma}
\begin{proof}
The first relation $[H_i\smallcup H_j]= 0$ is obvious since diagonal matrices commute.

The next relations $[U_i^\pm\smallcup U_j^\pm]\sim 0$ (same sign), follow directly if $j\not\in\{i+1,i-1\}$ as by \eqref{eq:bracked}:
\[
[E_{i}^{i+1}, E_{j}^{j+1}] = \delta_{j}^{i+1} E_i^{j+1} -\delta_i^{j+1} E_{j}^{i+1} =
\begin{cases}
0 & \text{ if $j\not\in\{i+1,i-1\}$,}\\
E_i^{i+2} & \text{ if $j= i+1$,}\\
-E_{i-1}^{i+1} & \text{ if $j= i-1$,}
\end{cases}
\]
and $H^2(\Gamma;\C_{\lambda_i/\lambda_{j}}) =0 $ for $|i-j|\geq 2$.

If $j=i+1$ then $u_i^+\smallcup u_{i+1}^+\in Z^2(\Gamma;\C_{\lambda_i/\lambda_{i+2}})=B^2(\Gamma;\C_{\lambda_i/\lambda_{i+2}})$ since
$\Delta_K(\lambda_i/\lambda_j)\neq 0 $ if $|i-j|\geq 2$. The same argument applies for $j=i-1$.
This proves $[U_i^+\smallcup U_j^+]\sim 0$. The proof of $[U_i^-\smallcup U_j^-]\sim 0$ is similar.

Now, consider  $[U_i^\pm,U_j^\mp]$ (opposite signs):
\[
[E_{i}^{i+1}, E_{j+1}^{j}] = \delta_{j+1}^{i+1} E_i^{j} -\delta_i^{j} E_{j+1}^{i+1} =
\begin{cases}
0 & \text{ if $j\neq i$,}\\
E_i^{i} -E_{i+1}^{i+1} & \text{ if $j=i$.}
\end{cases}
\]
Hence $[U_i^\pm,U_j^\mp] = 0$ if $i\neq j$, 
and $[U_i^\pm,U_i^\mp] = (u_i^+\smallcup u_i^-)(E_i^{i} -E_{i+1}^{i+1})$.   
Now, $u_i^+\smallcup u_i^- \sim 0$ since $H^2(\Gamma;\C)=0$. This imples $[U_i^\pm,U_i^\mp]\sim 0$.
\end{proof}

\begin{lemma}\label{lem:basis_H^2}
The $2(n-1)$ cocycles $[H_i\smallcup U_i^\pm]$, $i=1,\ldots,n-1$, represent a basis of
$H^2(\Gamma; {\Ad\rho_D})$.
\end{lemma}
\begin{proof}
By \eqref{eq:bracked} we obtain:
\begin{align*}
[H_i\smallcup U_i^+](\gamma_1,\gamma_2) &= h(\gamma_1) \big(\frac{\lambda_i}{\lambda_{i+1}} \big)^{h(\gamma_1)} u_i^+(\gamma_2)
[E_i^i - E_{i+1}^{i+1}, E^{i+1}_i] \\
&= 2 (h\smallcup u_i^+)(\gamma_1,\gamma_2) \,E_{i}^{i+1} \,.
\end{align*}
By a similar calculation we obtain $[H_i\smallcup U_i^-] = -2(h\smallcup u_i^-) E_{i+1}^i$.
The lemma  follows from the decomposition~\eqref{eq:decomp_sln}, Lemma~\ref{lem:cupproduct}, and equation~\eqref{eq:dim_sln}.
\end{proof}

\begin{lemma}\label{lem:HcupU}
For the generators \eqref{eq:generators} of $H^1(\Gamma;{\Ad\rho_D})$, and $\epsilon \in\{\pm\}$ we have: 
\begin{equation*}\label{eq:h_cup_u^+}
[H_i\smallcup U_j^\epsilon] =
\begin{cases}
0 & \text{ if $|i-j|>1$,}\\
-\epsilon (h\smallcup u_{i-1}^\epsilon) E_{i-1}^{i} & \text{ if $j=i-1$,}\\
2\epsilon (h\smallcup u_{i}^\epsilon) E_{i}^{i+1} & \text{ if $j=i$,}\\
-\epsilon (h\smallcup u_{i+1}^\epsilon) E_{i+1}^{i+2} & \text{ if $j=i+1$.}
\end{cases}
\end{equation*}
\end{lemma}
\begin{proof}
Direct calculation using equation~\eqref{eq:cupbil}.
\end{proof}

We are now ready to calculate $[U\smallcup U]$.
\begin{prop}\label{lem:quad_eq}
Let $U\in Z^1(\Gamma;{\Ad\rho_D})$ be a cocycle of the form:
\[
U =\sum_{i=1}^{n-1} x_i U_i^+ + y_i U_i^- + z_i H_i + C
\]
where $C\in B^1(\Gamma;{\Ad\rho_D})$ is a principal derivation, and $x_i,y_i,z_i\in\C$.

Then $[U\smallcup U]\sim 0$ if and only if for all $i=1,\ldots,n-1$:
\begin{equation}
\label{eq:ideal_I}
(2z_i - z_{i-1} -z_{i+1} ) x_{i} = 0 \quad\text{ and }\quad
(2z_i - z_{i-1} -z_{i+1} ) y_{i} = 0\,.
\end{equation}
\end{prop}

\begin{proof}
It follows from  Lemma~\ref{lem:HcupU} that
$[H_i\smallcup U_j^\pm]\sim 0$ if $|i-j|>1$. This and Lemma~\ref{lem:cup_trivial} give:
\begin{align*} 
[U\smallcup U ] & \sim
\sum_{i=1}^{n-1}   x_{i}z_{i-1} [H_{i-1}\smallcup U_i^+] + x_{i}z_{i} [H_i\smallcup U_i^+]+ x_{i}z_{i+1} [H_{i+1}\smallcup U_i^+] \notag\\
&\qquad + x_{i-1}z_{i} [H_{i}\smallcup U_{i-1}^+] + x_{i}z_{i} [H_i\smallcup U_i^+]+ x_{i+1}z_{i} [H_{i+1}\smallcup U_{i+1}^+]\\
&\qquad +y_{i}z_{i-1} [H_{i-1}\smallcup U_i^-] + y_{i}z_{i} [H_i\smallcup U_i^-]+ y_{i}z_{i+1} [H_{i+1}\smallcup U_i^-] \notag\\
&\qquad + y_{i-1}z_{i} [H_{i}\smallcup U_{i-1}^-] + y_{i}z_{i} [H_i\smallcup U_i^-]+ y_{i+1}z_{i} [H_{i+1}\smallcup U_{i+1}^-]\,.
\end{align*}
Here we have put $x_0=y_0=z_0=x_n=y_n=z_n=0$. With this convention we obtain:
\begin{equation}\label{eq:cup_product}
\frac{1}{2}[U\smallcup U ] \sim  \sum_{i=1}^{n-1}(2z_i - z_{i-1} -z_{i+1} ) \big(x_{i}\, [H_i\smallcup U_i^+] + y_{i} \, [H_i\smallcup U_i^-]\big)\,.
\end{equation}
Now, by Lemma~\ref{lem:basis_H^2}, the $2(n-1)$ cocycles $[H_i\smallcup U_i^\pm]$, $i=1,\ldots,n-1$, represent a basis of
$H^2(\Gamma; {\Ad\rho_D})$. This proves the proposition.
\end{proof}

\subsection{The ideal defining the quadratic cone}
\label{sec:ideal} 
In what follows we will use the following proposition:
\begin{prop} Let $\Gamma$ be a finitely generated group and $\rho\in R_n(\Gamma)$ a representation.
We fix a basis $(v_1,\ldots,v_k)$ of $T_\rho R_n(\Gamma) \cong Z^1(\Gamma;{\Ad\rho})$, and we let 
$(x_1,\ldots,x_k)$ denote the coordinates with respect to this basis. Then
the quadratic cone of $R_n(\Gamma)$ at $\rho$ is given by the following system of equations
\[
  \sum_{i,j=1}^k x_ix_j [v_i\smallcup v_j]\sim 0\,.
\]
\end{prop}
\begin{proof}
Following W.~Goldman a cocycle $u\in Z^1(\Gamma;{\Ad\rho})$, $u = \sum_{i=1}^k x_i v_i$,  is a curve tangent of second order if and only if
the $2$-cocycle (cup bracket) $[u\smallcup u]\sim 0$ is a coboundary. By Lemma~\ref{lem:2nd_curve_tangent} the curve tangents of second order coincide with the quadratic tangent cone. 
\end{proof}

Let $R$ denote the polynomial ring $R:= \C[x_1,\ldots,x_{n-1},y_1,\ldots,y_{n-1},z_1,\ldots,z_{n-1}]$.
In what follows we are interested in the ideal:
\begin{equation}\label{eq:ideal}
I^{(2)}= \big((2z_i - z_{i-1} -z_{i+1} )x_{i}, (2z_i - z_{i-1} -z_{i+1} )y_{i}\mid i=1,\ldots,n-1\big)
\end{equation}
which is generated  by the equations \eqref{eq:ideal_I}.

\begin{lemma}\label{lem:radical}
The ideal $I^{(2)}\subset R$ is a radical ideal i.e.\ $I^{(2)}$ is its own radical.
\end{lemma}
\begin{proof}
Consider the linear transformation $f$ of $\C^{3(n-1)}$ given by:
\[
f(x_i)=x_i,\quad f(y_i)=y_i,\quad\text{ and  }\quad f(z_i) = -z_{i-1}+2z_i - z_{i+1}
\]
where we have put $z_0=z_n=0$.
The linear transformation $f$ extends to an algebra isomorphism of
$R$ which maps the ideal $J$:
\[
J = (x_i z_i, y_i z_i \mid i=1,\ldots,n-1)\subset R
\]
to $I^{(2)}$. Notice that the ideal $J$ is a \emph{monomial ideal} i.e.\ $J$ is an ideal generated by monomials. Hence it follows from \cite[Corollary 1.2.5]{HerzogHibi} that $J$ is a radical ideal, and therefore $I^{(2)}$ also.
\end{proof}

The fact that $I^{(2)}$ is a radical ideal implies that $I^{(2)}$ is the intersection of finitely many prime ideals.
In order to describe these prime ideals we let $2^{n-1}$ denote the powerset of $\{1,\ldots,n-1\}$, and
for $\iota\in 2^{n-1}$ we denote by $\iota^\mathsf{c} := \{1,\ldots,n-1\}\smallsetminus \iota$ the complement of $\iota$. The size of $\iota\in 2^{n-1}$ will be denoted by $|\iota|$.
We are interested in the ideals $I_\iota$, $\iota\in 2^{n-1}$, given by:
\[
I_\iota = (2z_i - z_{i-1} -z_{i+1}, x_j, y_j \mid i\in\iota\text{ and } j\in \iota^\mathsf{c})\,.
\]
Notice that $I_\iota\subset R$ is a prime ideal. In fact $I_\iota$ is generated by linear equations, and 
$R/I_\iota$ is isomorphic to a polynomial algebra.
The scheme
\[
V_\iota := \big(R/I_\iota,\Spec(R/I_\iota)\big)
\subset\mathbb{A}^{3(n-1)}
\]
is an affine subspace of $\mathbb{A}^{3(n-1)}$ given by the equations
\begin{equation} \label{eq:V_iota}
\begin{cases}
2z_i - z_{i-1} -z_{i+1} =0 & \text{ for $i\in \iota$,} \\
x_j = y_j =0 &\text{ for $j\in \iota^\mathsf{c}$.}
\end{cases}
\end{equation}
Hence $V_\iota$ is a reduced scheme and 
\begin{equation} \label{eq:dimViota}
\dim V_\iota = 3(n-1) - |\iota| - 2|\iota^\mathsf{c}| = 3(n-1) - |\iota| - 2\big((n-1) -|\iota|\big)=n-1 + |\iota|\,.
\end{equation}

Given a set $E\subset R$ we will denote its \emph{zero locus} by $V(E)\subset\mathbb{A}^{3(n-1)}$.

\begin{lemma} \label{lem:prime_dec}
The radical ideal $I^{(2)}$ is the intersection of the prime ideals $I_\iota$:
\[
I^{(2)} = \displaystyle\bigcap_{\iota\in 2^{n-1}} I_\iota\,.
\]
\end{lemma}
\begin{proof}
Since $I^{(2)}$ and $\bigcap_{\iota\in2^{n-1}} I_\iota$ are both radical ideals it will be sufficient to prove that
\[ 
V(I^{(2)}) = V\Big(\bigcap_{\iota\in2^{n-1}} I_\iota\Big) = \bigcup_{\iota\in2^{n-1}} V(I_\iota)\,.
\]

If $S$ and $T$ are two finite sets of polynomials then $V(S)\cup V(T)=V(S\,T)$, where $S\,T$ denotes the set of all products of an element of $S$ by an element of $T$.
Therefore, we obtain:
\[
V\big((2z_i - z_{i-1} -z_{i+1} )x_{i}, (2z_i - z_{i-1} -z_{i+1} )y_{i} ) = V(x_{i}, y_{i})\cup V(2z_i - z_{i-1} -z_{i+1} )\,.
\]
Hence
\begin{align*}
V(I^{(2)}) &= \bigcap_{i=1}^{n-1} \big( V(x_{i}, y_{i})\cup V(2z_i - z_{i-1} -z_{i+1} ) \big)\\
& =  \bigcup_{\iota\in 2^{n-1}} \Big(\big( \bigcap_{i\in\iota} V(2z_i - z_{i-1} -z_{i+1})\big) \cap 
\big( \bigcap_{j\in\iota^c} V( x_j, y_j)\big)\Big) \\
         & =  \bigcup_{\iota\in 2^{n-1}}  V(2z_i - z_{i-1} -z_{i+1}, x_j, y_j \mid i\in\iota\text{ and } j\in \iota^\mathsf{c})\\
         & =  \bigcup_{\iota\in 2^{n-1}} V(I_\iota)
\end{align*}
which proves the lemma.
\end{proof}

\subsection{The quadratic cone at the diagonal representation}
\label{sec:cone}

In order to describe the quadratic cone at $\rho_D$ we choose the basis of
$Z^1(\Gamma; {\Ad\rho_D})$ given by the non principal derivations  $ H_i$, $U_i^\pm$, $i=1,\ldots,n-1$  from equation~\eqref{eq:generators}, 
and the $(n^2-n)$ principal derivations $B_k^l$, $1\leq k,l\leq n$, $k\neq l$, given by  
\[
B_k^l(\gamma) = \big( (\lambda_k/\lambda_l)^{h(\gamma)} -1\big) E_k^l\,.
\]

By means of this basis we obtain an identification 
$\mathbb{A}^{n^2+2n-3}\cong Z^1(\Gamma;{\Ad\rho_D})$, and we let 
$(x_i, y_i, z_i, t_k^l)$, $i=1,\ldots,n-1$, $1\leq k,l\leq n$, $k\neq l$, denote the coordinates of 
$\mathbb{A}^{n^2+2n-3}$. 

In what follows we will identify $\C[x_i, y_i, z_i]\otimes\C[t_k^l]$ with $\C[x_i, y_i, z_i,t_k^l]$. This corresponds to the identification 
$\mathbb{A}^{3n-3}\times \mathbb{A}^{n^2-n}=\mathbb{A}^{n^2+2n-3}$.
Moreover, we let  $\widetilde I^{(2)}$ and $\widetilde I_\iota$ denote the ideals
\[
\widetilde I^{(2)}:= I^{(2)}\otimes \C[t_k^l]\subset \C[x_i, y_i, z_i,t_k^l] \quad\text { and }\quad
\widetilde I_\iota := I_\iota\otimes \C[t_k^l]\subset \C[x_i, y_i, z_i, t_k^l] 
\]
respectively. We will write  $\widetilde V_\iota := V(\widetilde I_\iota)
\subset \mathbb{A}^{n^2+2n-3}$ for short.

\begin{prop}\label{prop:quad_cone}
The quadratic cone $TQ_{\rho_D}R_n(\Gamma)\subset T_{\rho_D}R_n(\Gamma)$ is a union
of affine subspaces, and  under the identification
$ T_{\rho_D} R_n(\Gamma)\cong Z^1(\Gamma;{\Ad\rho_D})\cong \mathbb{A}^{n^2+2n-3}$, we obtain:
\[
TQ_{\rho_D}R_n(\Gamma) = V(\widetilde I^{(2)})=\bigcup_{\iota\in 2^{n-1}}\widetilde  V_\iota \,.
\]
In particular, as the ideal $\widetilde I^{(2)}$ is radical, the quadratic cone $TQ_{\rho_D}R_n(\Gamma)$ is reduced.
\end{prop}

\begin{proof} First we will show that if an ideal $I\subset \C[x_i, y_i, z_i]$ is radical, then
$I\otimes\C[t_k^l] \subset \C[x_i, y_i, z_i]\otimes \C[t_k^l]$ is also radical.
To this end, notice that  for any ideal $I\subset \C[x_i, y_i, z_i]$ we have
\[
(\C[x_i, y_i, z_i]\otimes \C[t_k^l])/ (I\otimes\C[t_k^l])  \cong (\C[x_i, y_i, z_i]/I) \otimes \C[t_k^l]\,.
\]
Moreover, the tensor product of two reduced $\C$-algebras is reduced \cite[V,\S15,Thm~3]{Bourbaki}. This implies that
if $I\subset \C[x_i, y_i, z_i]$ is radical, then $I\otimes\C[t_k^l]\subset \C[x_i, y_i, z_i] \otimes \C[t_k^l]$ is also radical. Consequently the ideals $\widetilde I_\iota$ are radical, and
\[
\widetilde V_\iota = V(\widetilde I_\iota) = V(I_\iota)\times \mathbb{A}^{n^2-n} 
\subset \mathbb{A}^{3n-3}\times \mathbb{A}^{n^2-n}=\mathbb{A}^{n^2+2n-3}
\] is an affine subspace.

For two closed subschemes $X_k=(R/I_k, \Spec(R/I_k))$, $k = 1, 2$, of $\mathbb{A}^N$
their union is defined by:
\[
X_1\cup X_2 = \big(R/(I_1\cap I_2), \Spec(R/(I_1\cap I_2))\big)
\]
(see \cite[Section I.\S4]{Kunz}  and \cite[Section I.2.1]{EisenbudHarris}). 
As the intersection of radical ideals is radical, the  assertion follows from Lemma~\ref{lem:prime_dec}.
\end{proof}

In what follows, given an element $\iota\in 2^{n-1}$, 
we will identify the linear subspace of 
$Z^1(\Gamma;{\Ad\rho_D})$ defined by the radical ideal 
$\widetilde I_\iota$ with the affine subspace $\widetilde V_\iota$. Notice that by equation~\eqref{eq:dimViota} we have:
\[
\dim\widetilde V_\iota = n^2-n + (n-1 +|\iota|) = n^2-1 +|\iota|\,.
\]
Therefore, there is exactly one component $\widetilde V_\emptyset$ of dimension $n^2-1$ and exactly one component $\widetilde V_{\{1,\ldots,n-1\} }$ of dimension $n^2+n-2$.
These components are of minimal, respectively, maximal dimension.
If $n=2$ then there are no other components, but for $n>2$ there are.

We can easily identify the component of minimal dimension $n^2-1$:
the surjection $h\co\Gamma\to\Z$ induces a closed immersion (see Proposition 1.7 of \cite{Lubotzky-Magid1985})
$h^*\co R_n(\Z)\hookrightarrow R_n(\Gamma)$. Its image $h^*(R_n(\Z))$ is the subvariety of abelian representations of $\Gamma$ and will be denoted by $R_{ab}$. Notice that $R_{ab}\cong\SLn$ is an irreducible 
$(n^2-1)$-dimensional component of $R_n(\Gamma)$.

\begin{prop} \label{prop:R^n}
Let $\widetilde V_{\{1,\ldots,n-1\}}$ and  $\widetilde V_\emptyset$  be the components of extreme dimension in the quadratic cone
$TQ_{\rho_D}R_n(\Gamma)$. Then
\[
T_{\rho_D} R^{(n)}_D  =\widetilde V_{\{1,\ldots,n-1\}}
\quad
\text{ and }
 \quad
T_{\rho_D} R_{ab}  =\widetilde V_\emptyset\,.
\]
Therefore, $\rho_D$ is a smooth point of each of the two components
$R^{(n)}_D$ and $R_{ab}$ which intersect transversely at the orbit of $\rho_D$.
\end{prop}
\begin{proof}
The irreducible algebraic component $R_{ab}$  has dimension $n^2-1$, and is isomorphic to $\SLn$.
Hence  $\rho_D\in R_{ab}$ is a smooth point. A tangent vector in $T_{\rho_D} R_{ab}$ is a linear combination of $\{H_1,\ldots, H_{n-1}\}$ and $\{B_k^l\}$ therefore
$T_{\rho_D} R_{ab}$ corresponds to a subspace of~$\widetilde V_\emptyset$.
Now, $T_{\rho_D} R_{ab} = \widetilde V_\emptyset$ follows since
\[
\dim \widetilde V_\emptyset = n^2-1 = \dim T_{\rho_D} R_{ab}\,.
\]

By Lemma~\ref{lem:upper-triang-exist},  the representation $\rho_D$ is contained in the closure of the orbit of $\rho^{(n)}_D$ and hence is contained in the component $R^{(n)}_D$ of maximal dimension $n^2+n-2$. 

The tangent cone $TC_{\rho_D} R^{(n)}_D$ has only components of dimension $(n^2+n-2)$, and is contained in
the quadratic cone $TQ_{\rho_D} R^{(n)}_D$. 
On the other hand,  $TQ_{\rho_D} R^{(n)}_D$ is contained in $TQ_{\rho_D} R_n(\Gamma)$, and therefore 
$TC_{\rho_D} R^{(n)}_D$ is contained in $\widetilde V_{\{1,\ldots,n-1\}}$ since this  is the unique 
$(n^2+n-2)$-dimensional component  of maximal dimension of $TQ_{\rho_D} R_n(\Gamma)$ (see Lemma~\ref{prop:quad_cone}). Hence
$TC_{\rho_D} R^{(n)}_D= \widetilde V_{\{1,\ldots,n-1\}}$ is a linear subspace.
This implies that the tangent cone $TC_{\rho_D}R^{(n)}_D$ is an affine space
hence it coincides with the tangent space $TC_{\rho_D} R^{(n)}_D$.
It follows that $\rho_D\in R^{(n)}_D$ is a smooth point.
Finally, for $\widetilde V_\emptyset$ and $\widetilde V_{\{1,\ldots,n-1\}}$ we have
\begin{align*}
 \dim (\widetilde V_{\{1,\ldots,n-1\}}+ \widetilde V_\emptyset) &= \dim(\widetilde V_{\{1,\ldots,n-1\}}) + \dim(\widetilde V_\emptyset) - \dim B^1(\Gamma;\Ad\rho)\\
 & = n^2+n-2 + n^2 -1 - (n^2 -n)\\
 & = n^2 + 2n -3 = \dim T_{\rho_D} R_n(\Gamma)\,.\qedhere
\end{align*}

\end{proof}

\subsection{The tangent cone at the diagonal representation}
\label{sec:tangent_cone}

In this section, we will prove that every cocycle in the quadratic cone at ${\rho_D}$ is an integrable cocycle. More precisely, we construct an arc of representations with endpoint ${\rho_D}$ and tangent to a given cocycle in the quadratic cone.

\begin{thm}\label{thm:tangent-cone}
Every component $\widetilde V_\iota$ is contained in the tangent cone $TC_{\rho_D}R_n(\Gamma)$.
Therefore, the tangent cone $TC_{\rho_D}R_n(\Gamma)$ and quadratic cone 
$TQ_{\rho_D}R_n(\Gamma)$ coincide. Hence the tangent cone $TC_{\rho_D}R_n(\Gamma)$ is reduced, and consequently  the diagonal representation $\rho_D$ is a reduced point of the representation variety.
\end{thm}

In what follows we let $A\oplus B$ denote the direct sum of the matrices $A$ and $B$:
\[
 A\oplus B =
 \left(
 \begin{array}{c|c}
  A & \mathbf{0}\\
  \hline
  \mathbf{0} &B
 \end{array}
 \right)
\,.
\]

\begin{example}[The case $n=3$]\label{exe:n=3}
Before treating the general case, we shall consider the case $n=3$ i.e.\ 
 $D=\operatorname{diag}(\lambda_1,\lambda_2,\lambda_3)$. 
 
 As in Theorem~\ref{thm:upper-triang-deform} we suppose that $\lambda_1/\lambda_2$ and $\lambda_2/\lambda_3$ are simple roots of the Alexander polynomial $\Delta_K$, and that $\lambda_1/\lambda_3$ is not a root of $\Delta_K$.
 By Proposition~\ref{prop:quad_cone} the quadratic cone $TQ_D R_3(\Gamma)$ has four components
 $\widetilde V_\emptyset$, $\widetilde V_{\{1\}}$, $\widetilde V_{\{2\}}$ and $\widetilde V_{\{1,2\}}$. 
 Let us consider these components :
 \begin{description}
  \item[$\widetilde V_\emptyset$] The component $\widetilde V_\emptyset$ is $8$-dimensional. Up to principal derivations the vector space $\widetilde V_\emptyset$ is  generated by the two cocycles $H_1=h(E_1^1-E_2^2)$ and $H_2=h(E_2^2-E_3^3)$. Each linear combination of these cocycles is integrable and
   \[
   \rho^{(c_1H_1+c_2H_2)}_t(\gamma) = \exp\big(t(c_1H_1+c_2H_2)(\gamma)\big)\rho_D(\gamma)
  \]
 is clearly an arc of representations as diagonal matrices commute.
 
  \item[$\widetilde V_{\{1,2\}}$] This component is a $10$-dimensional vector space.
  Up to principal derivations it is generated by the four cocycles $U^{\pm}_k$ for $k=1,2$. By Proposition~\ref{prop:R^n} the representation $\rho_D$ is a smooth point of the component $R^{(3)}_D$, and $\widetilde V_{\{1,2\}}$ is the tangent space $T_{\rho_D} R^{(3)}_D$. Hence every vector in $\widetilde V_{\{1,2\}}$ is integrable.
  
  \item[$\widetilde V_{\{1\}}$] This component is a $9$-dimensional vector space which is generated, up to coboundaries,    
  by the three cocycles 
  $H:=H_1+2H_2= \Big(\begin{smallmatrix} 
   h & 0 & 0 \\
   0 &  h & 0\\
   0 & 0 &-2h
   \end{smallmatrix}\Big)$ 
  and by $U^{\pm}_1$. In order to show that every cocycle in $\widetilde V_{\{1\}}$ is integrable we have to show that the linear combination 
  \[
   c H + d^+ U_1^+ + d^-U_1^- =
   \begin{pmatrix} 
   ch & d^+ u_1^+ & 0 \\
   d^- u_1^- & c h & 0\\
   0 & 0 &-c2h
   \end{pmatrix}\,,\text{ $c, d^+, d^-\in\C$, }
  \]
  is integrable. 
  
   We will see that the cocycle 
  \[
   U =
   \begin{pmatrix} 
   0 & d^+ u_1^+ & 0 \\
   d^- u_1^- & 0 & 0\\
   0 & 0 &0
   \end{pmatrix}
   \]
   is integrable in a specific way.
  To this end, we choose a complex number $z_1$ such that $z_1^2=\lambda_1\lambda_2$, and we consider the decomposition of the matrix $D = z_1 D_1 \oplus z_2 D_2$ where 
  $D_1 = \operatorname{diga}(\lambda_1/z_1,\lambda_2/z_1)\in\SLn[2]$, $z_2 = \lambda_3$, and $D_2=(1)\in\SLn[1]$.
  Now, $\rho_{D_1}\co\Gamma\to\SLn[2]$ is a representation which verifies the hypothesis of Theorem~\ref{thm:main1}, and 
  the cocycle 
  $U_1 := \big(\begin{smallmatrix} 0 & d^+ u_1^+ \\ d^- u_1^- & 0\end{smallmatrix}\big)$ is integrable i.e. there exist an analytic path of representations $\rho^{(U_1)}_{t}\co\Gamma\to\SLn[2]$ such that $\rho^{(U_1)}_{0} = \rho_{D_1}$, and 
  \[
   \frac{d }{dt} \rho^{(U_1)}_t(\gamma) \big|_{t=0}\,
   \rho_{D_1}(\gamma)^{-1} = U_1(\gamma) = 
   \begin{pmatrix} 
   0 & d^+ u_1^+(\gamma) \\ 
   d^- u_1^-(\gamma) & 0
   \end{pmatrix}
  \]
  for all $\gamma\in\Gamma$.
  Finally we define a path of representations $\rho^{(cH+U)}_{t}\co\Gamma\to\SLn[3]$ by
  \[
   \rho^{(cH+U)}_{t} (\gamma) = \exp\big(tcH(\gamma)\big)
   \big(
   z_1^{h(\gamma)} \rho^{(2)}_t(\gamma) \oplus z_2^{h(\gamma)}
   \big) \,.
  \]
  Notice that $\rho^{(cH+U)}_{t}$ is a representation since the block matrices commute. Moreover,
  \[
   \rho^{(cH+U)}_{0} (\gamma) = 
  I_3 \big(
   z_1^{h(\gamma)} \rho_{D_1}(\gamma) \oplus z_2^{h(\gamma)}
   \big)  = \begin{pmatrix} z_1
                    \begin{pmatrix} 
                    \lambda_1/z_1 & 0 \\
                    0 & \lambda_2/z_1
                    \end{pmatrix} & \begin{matrix} 0 \\ 0 \end{matrix} \\
                0 \qquad 0  & z_2
               \end{pmatrix}^{h(\gamma)} = \rho_D(\gamma)\,,
  \]
  and 
  \[
  \frac{d }{dt} \rho^{(cH+U)}_t(\gamma) \big|_{t=0}\,
   \rho_{D}(\gamma)^{-1} = \Big(
 c H(\gamma) + \big(U_1(\gamma) \oplus 0 \big)\Big)   
   = (c H + U)(\gamma)\,.  
  \]                                                                  
  \item[$\widetilde V_{\{2\}}$] This case is similar to the previous case.
  \end{description}
\end{example}

In order to describe the general case we need some notations.
We describe the cocycles in $\widetilde V_\iota$ for a proper subset
$\iota\in 2^{n-1}$, $\iota\neq\emptyset$, $\iota\neq\{1,\ldots,n-1\}$. For such   $\iota$ we let denote
$\iota^\mathsf{c}=\{j_1,\ldots,j_l\}\subset\{1,\ldots,n-1\}$ where 
$0=j_0< j_1< j_2<\cdots< j_l <j_{l+1} = n$, and
we define $n_s := j_s - j_{s-1}$, $s=1,\ldots,l+1$.
Notice that $\sum_{s=1}^{l+1} n_s = n$.

Now, equations~\eqref{eq:V_iota} give that for each $i\in\iota$ we have $U_i^\pm\in \widetilde V_\iota$,
and for $s=1,\ldots,l$ we have $F_{n_s} \in \widetilde V_\iota$ where $F_{n_s}$ denotes the diagonal matrix
\[
F_{n_s} = \operatorname{diag}(\underbrace{0,\ldots,0}_{n_1+\cdots+n_{s-1}},
\underbrace{ n_{l+1} h }_{n_s}, \underbrace{0,\ldots,0}_{n_{s+1}+\cdots+n_{l}},
\underbrace{-n_s h}_{n_{l+1}})\,.
\]

Therefore,
\[
\widetilde V_\iota = \langle U_i^\pm,\ F_{n_s} \mid i\in\iota,\ s=1,\ldots,l \rangle \oplus B^1(\Gamma;{\Ad\rho_D})\,.
\]
Notice that we have, for all $U^\pm_i$ with $i\in\iota$ and all
$F_{n_s}$, that
$ F_{n_s}(\gamma)$ and $U_i^\pm(\gamma)$ commute for all $\gamma\in\Gamma$.
Then a cocycle $U\in \widetilde V_\iota$ can be written as a linear combination of 
the diagonal matrices $F_{n_s}$, and
block matrices:
\begin{equation}\label{eq:dec_U}
U=\begin{pmatrix}
U_{n_1}&&\\
0&\ddots&0\\
&&U_{n_{l+1}}
\end{pmatrix}=
\bigoplus_{s=1}^{l+1} U_{n_s}
\end{equation}

where for all $1\leq s\leq l+1$:
\[
U_{n_s}=\begin{pmatrix}
0&u^+_{j_{s-1}+1}&0&&\ldots&0\\
u^-_{j_{s-1}+1}&0&u^+_{j_{s-1}+2}&&\ddots&\vdots\\
0&u^-_{j_{s-1}+2}&\ddots&&\ddots&0\\
\vdots&\ddots&\ddots&&\ddots&u^+_{j_{s}-1}\\
0&\ldots&0&&u^-_{j_{s}-1}&0
\end{pmatrix}\,.
\]
Here,
$ u_k^+\in Z^1(\Gamma;\C_{\lambda_k/\lambda_{k+1}})$ and $u_{k}^-\in Z^1(\Gamma;\C_{\lambda_{k+1}/\lambda_k})$ where $j_{s-1}+1\leq k\leq j_s-1$.

\begin{proof}[Proof of Theorem~\ref{thm:tangent-cone}]%
The cases $\iota=\emptyset$, and $\iota=\{1,\ldots,n-1\}$ were treated in Proposition~\ref{prop:R^n}.
So from now on we will suppose that $\iota$ is a proper subset of $2^{n-1}$.

First we notice that every diagonal cocycle is integrable, and for $F=\sum_{s=1}^l c_s F_{n_s}\in \widetilde V_\iota$ we define
\[
 \rho_t^{(F)}(\gamma) = \exp\big(t F(\gamma) \big)\,.
\]

For the non diagonal cocycles we proceed as in Example~\ref{exe:n=3}.
Let $U\in \widetilde V_\iota$ be a given cocycle.
We will use the decomposition \eqref{eq:dec_U}
\[
U = \bigoplus_{s=1}^{l+1} U_{n_s}\, ,
\]
to split up a diagonal matrix 
$D=D(\lambda_1,\ldots\lambda_n)\in\SLn$ accordingly as $D = \tilde D_1\oplus\cdots\oplus \tilde D_{l+1}$, where
$\tilde D_s = D(\lambda_{j_{s-1}+1},\ldots, \lambda_{j_{s}})$ is build from $n_s$ consecutive diagonal entries of $D$. We choose $z_s\in\C^*$ such that $z_s^{n_s} = \det (\tilde D_s)$, and we define
$D_s : = z_s^{-1} \, \tilde D_s \in \SLn[n_s]$. We obtain 
\[
 D = z_1 D_1 \oplus\cdots\oplus z_{l+1} D_{l+1}\,.
\]
Now, $\rho_{D_s}$ verifies the hypotheses of Lemma~\ref{lem:upper-triang-exist} and
Theorem~\ref{thm:upper-triang-deform}, and therefore there is a $(n_s^2+n_s-2)$-dimensional irreducible component $R^{(n_s)}_{D_s}\subset R_{n_s}(\Gamma)$.
By Lemma~\ref{lem:tgt-space-variety} we have
\[
 T_{\rho_{D_s}}R^{(n_s)}_{D_s} = \langle U^\pm_{l}\ |\ j_{s-1}+1\leq l\leq j_{s}-1\rangle
 \oplus B^1(\Gamma;{\Ad\rho_{D_s}})\,.
\]

Notice that in the case $n_s = 1$ we have $D_s = (1)\in\SLn[1]=\{(1)\}$ is the trivial group and $R^{(1)}_{D_s}$ is zero-dimensional and contains only the trivial representation.

We define representations $\rho_t^{(U)}\co\Gamma\to\SLn$ by
\[
 \rho_t^{(U)}(\gamma) = \bigoplus_{s=1}^{l+1} z_s^{h(\gamma)} 
 \rho^{(U_{n_s})}_t(\gamma)
\]
where $\rho^{(U_{n_s})}_t\co \Gamma\to \SLn[n_s]$ is an arc of representations such that
\[
 \frac{d}{dt} \rho^{(U_{n_s})}_t(\gamma)\big|_{t=0}\cdot \rho_{D_{s}}(\gamma)^{-1} = U_{n_s}(\gamma)\,.
\]
Notice that $\rho_0^{(U)} = \rho_D$, and it follows directly that
\[
  \frac{d}{dt} \rho^{(U)}_t(\gamma)\big|_{t=0}\cdot \rho_{D}(\gamma)^{-1} = U(\gamma)\,.
\]

Finally notice that the matrices $\rho_t^{(U)}(\gamma)$ and $\rho_t^{(F)}(\gamma)$ commute for all $F$ and $U$ in $\widetilde V_\iota$.
Therefore, $\rho_t^{(F+U)}\co\Gamma\to\SLn$ given by
\[
 \rho_t^{(F+U)}(\gamma) = \rho_t^{(F)}(\gamma)\cdot \rho_t^{(U)}(\gamma)
\]
is a path of representation, and its derivative in $t=0$ gives the cocyle $F+U \in \widetilde V_\iota$.

The last assertion is an immediate consequence of Proposition~\ref{prop:quad_cone} and Lemma~\ref{lem:scheme2}.
\end{proof}

\begin{remark}\label{rem:red_rep}
Notice that for a proper subset $\iota$  of $2^{n-1}$ all the representations $\rho_t^{(F+U)}$ are included in $R^\mathit{red}_n(\Gamma)$ the closed subscheme of reducible representations (\cite[Lemma~7.7]{Heusener-Porti2015}).
\end{remark}

\begin{thm}\label{thm:components}
For each $k$, $0\leq k \leq n-1$, there is at least one
irreducible algebraic component of $R_n(\Gamma)$ of dimension $n^2-1+k$ which contains $\rho_D$. More precisely:
\begin{enumerate}
    \item\label{item1} The component $R_D^{(n)}$ is the unique component which contains $\rho_D$ and is of dimension $(n^2+n-2)$.
    \item\label{item2} The component $R_n(\Z)$ is the unique component which contains $\rho_D$ and is of dimension $n^2-1$.
    \item\label{item3} The component $R_D^{(n)}$ is the unique component that contains $\rho_D$ and irreducible representations.
       
\end{enumerate}
\end{thm}

\begin{proof}
In the proof  we will make use of the following fact which follows directly from the definitions: if $V$ and $W$ are two subschemes of $\mathbb{A}^N$, and $p\in V\cap W$ is a point then
 \begin{equation}\label{eq:intersectionOfCones}
 TC_p(V) \cup  TC_p(W)\subset TC_p(V\cup W)\,.
 \end{equation}

To prove (\ref{item1}) we let $\widetilde R_{D}$ denote an $(n^2+n-2)$-dimensional algebraic component of $R_n(\Gamma)$ which contains $\rho_D$.
   The tangent cone $TC_{\rho_D} \widetilde R_{D}$ has only components of dimension $(n^2+n-2)$, and is contained in
the quadratic cone $TQ_{\rho_D} \widetilde R_{D}$ 
which is contained in $TQ_{\rho_D} R_n(\Gamma)$. Therefore 
$TC_{\rho_D} \widetilde R_{D}$ is contained in $\widetilde V_{\{1,\ldots,n-1\}}$ since this  is the unique 
$(n^2+n-2)$-dimensional component of maximal dimension of $TQ_{\rho_D} R_n(\Gamma)$ (see Lemma~\ref{prop:quad_cone}). Hence
$TC_{\rho_D} \widetilde R_{D}= \widetilde V_{\{1,\ldots,n-1\}}$ is a linear subspace.
This implies that the tangent space $T_{\rho_D} \widetilde R_{D}$ coincides with the tangent cone $TC_{\rho_D} \widetilde R_{D}$, and $\rho_D\in \widetilde R_{D}$ is a smooth point.    
 On the other hand,
      \[
\widetilde V_{\{1,\ldots,n-1\}}=TC_{\rho_D}\widetilde R_{D}\cup TC_{\rho_D}R_{D}^{(n)}
\subset TC_{\rho_D}(\widetilde R_{D}\cup R_{D}^{(n)})\subset \widetilde V_{\{1,\ldots,n-1\}}\;.
            \]
The last inclusion is consequence of the fact that the tangent cone  $TC_{\rho_D}R_n(\Gamma)$ contains only one component of maximal dimension $n^2+n-2$.
Thus, we obtain equality and $\rho_D$ is a smooth point of $\widetilde R_{D}\cup R_{D}^{(n)}$. This is possible only if 
$\widetilde R_{D}=R_{D}^{(n)}$.

The proof of (\ref{item2}) is similar.

The proof of  (\ref{item3}) is a consequence of Lemma \ref{lem:reducedpoints}. More precisely, suppose that $V$ is a component of $R_n(\Gamma)$ that contains $\rho_D$ and an irreducible representation $\varrho$. If $\dim V< n^2+n-2$ we will obtain a contradiction by Lemma \ref{lem:reducedpoints} since $H^0(\Gamma;\Ad\varrho) = 0$. Thus $\dim V\geq n^2+n-2$. As $R_{D}^{(n)}$ is the only component of $R_n(\Gamma)$ that contains $\rho_D$ and is of maximal dimension, we conclude that $V=R_{D}^{(n)}$.

\end{proof}

\section{Examples}\label{sec:examples}
In this paragraph we discuss and illustrate by some examples the significance of our results in light of what is already known in the literature.

\subsection{Number of components}\label{sec:number_components}
 For each $k$, $0\leq k \leq n-1$, there is at least one
irreducible algebraic component of $R_n(\Gamma)$ of dimension $n^2-1+k$ which contains $\rho_D$. 
So there are at least~$n$ irreducible components of $R_n(\Gamma)$ which contain 
$\rho_D$. An example is the torus knot $T(3,2)$, and $n=3$. The representation variety $R_3(\Gamma_{3,2})$ has exactly three algebraic components  (see \cite[Theorem~1.1]{Munoz-Porti2016}). All the components contain the diagonal representation $\rho_D$ where $D= \operatorname{diag}( \exp(\pi/3), 1, \exp(-\pi/3) )$ (see Subsection~\ref{example:T(3,2)}).

By Theorem~\ref{thm:components}, there exist a unique component of maximal dimension $n^2+n-2$ which contains $\rho_D$.
We conjecture that there are at most $\binom{n-1}{k}$ components of dimension $n^2-1+k$ which contain $\rho_D$. 
Hence we expect that are at most $2^{n-1}$ components through $\rho_D$.
An example is given by the torus knot $T(3,4)$, and $n=3$. The representation variety $R_3(\Gamma_{3,4})$ has exactly two algebraic components of dimension $9$ \cite[Theorem~1.1]{Munoz-Porti2016}. With help of \cite[Proposition~9.1]{Munoz-Porti2016} it is possible to prove that the diagonal representation of $\Gamma_{3,4}$ which maps the meridian to
$D=\operatorname{diag}(\xi^4 , \xi ,\xi^{-5})$ for $\xi=\exp(i\pi/18 )$ is contained in both components. Moreover, the representation $\rho_D$ verifies the hypothesis of Theorem~\ref{thm:upper-triang-deform}, and therefore $\rho_D$ is contained in $4$ components. (See the discussion in Subsection~\ref{example:T(3,4)}.)

\subsection{Knots with quadratic Alexander polynomial}
For $\SLn[3]$ and knots with a quadratic Alexander polynomial, Theorem~\ref{thm:upper-triang-deform} follows from the results in \cite{Heusener-Medjerab2014}. More precisely, up to an action of the center $\{I_3,\omega I_3, \omega^2 I_3\} $ of $\SLn[3]$, every diagonal representation which verifies the conditions of Theorem~\ref{thm:upper-triang-deform} is the composition of a diagonal representation in $\SLn[2]$ with the canonical representation $r_3\co\SLn[2]\to\SLn[3]$. To see this, let $K$ be such a knot with a root $\alpha$, and let $\rho_D$ be an abelian representation of $K$ which maps the meridian to 
\[
\begin{pmatrix}
\lambda_1 & 0&0\\
0&\lambda_2&0\\
0 &0& \lambda_3
\end{pmatrix}\in\SLn[3]
\] with $\lambda_i,\ 1\leq i\leq 3$ pairwise distinct. Then 
\[
\begin{cases}
\lambda_1\lambda_2\lambda_3=1\\
\lambda_1/\lambda_2=\alpha^{\epsilon_1};\ \epsilon_1=\pm1\\
\lambda_2/\lambda_3=\alpha^{\epsilon_2};\ \epsilon_2=\pm1
\end{cases}
\]
We obtain that $\lambda_1=\lambda_3\alpha^{\epsilon_1+\epsilon_2}$, $\lambda_2=\lambda_3\alpha^{\epsilon_2}$, $\epsilon_1,\epsilon_2=\pm1$ and $\lambda_3^3\alpha^{\epsilon_1+2\epsilon_2}=1$. The condition $\lambda_1\not=\lambda_3$ implies that $\epsilon_1=\epsilon_2$ and $\lambda_3=\alpha^{-\epsilon_1}\omega^k$ with $\omega$ a primitive 3rd root of unity. As a consequence 
\[
\rho_D(\mu)=\omega^k
\begin{pmatrix}
\alpha^{\epsilon_1} & 0&0\\
0&1&0\\
0 &0& \alpha^{-\epsilon_1}
\end{pmatrix} =
\omega^k\,
r_3 \begin{pmatrix}
\alpha^{\epsilon_1/2} & 0\\
0& \alpha^{-\epsilon_1/2}
\end{pmatrix}\,.
\]
This proves the claim.

\subsection{The torus knot $T(3,4)$}\label{example:T(3,4)}
On the other hand, if the Alexander polynomial is not quadratic, there might be deformations which can not be detected by the results in \cite{Heusener-Medjerab2014}. To see this let us consider the torus knot $T(3,4)$ which is the knot $8_{19}$ in the knot tables. The knot group is $\Gamma_{3,4} = \langle a,b \mid a^3 = b^4\rangle$, a meridian is given by $m= ab^{-1}$, and the Alexander polynomial is $\Delta_{3,4}(t) = (1 - t + t^2)  (1 - t^2 + t^4)$ is the product of two cyclotomic polynomials $\phi_6(t)$ and $\phi_{12}(t)$. The abelianization $h\co\Gamma_{3,4}\to\langle t\mid -\rangle$ is given by
$h(a) = t^4$ and $h(b) =t^3$. 
The abelian  representation of $\Gamma_{3,4}$ which maps the meridian to
\footnotesize
\[
 \begin{pmatrix} \lambda_1 & 0 &0 \\ 0 & \lambda_2 & 0\\ 0 & 0 &\lambda_3\end{pmatrix}
=  \begin{pmatrix} \xi^4 & 0 &0 \\ 0 & \xi & 0\\ 0 & 0 &\xi^{-5}\end{pmatrix} 
\]
\normalsize
for $\xi=\exp(i\pi/18 )$
verifies the hypothesis of Theorem~\ref{thm:upper-triang-deform}: 
$\lambda_1/\lambda_2 =\xi^3$ is a primitive 12-th root of unity, $\lambda_2/\lambda_3 =\xi^6$
is a primitive 6-th root of unity, and $\lambda_1/\lambda_3 =\xi^9$  is not a root of $\Delta_{3,4}(t)$. 
Therefore by Theorem~\ref{thm:upper-triang-deform} the abelian representation admits irreducible deformations, and is contained in a 10-dimensional component of $R_3(\Gamma_{3,4})$. These deformations can not be detected by the results in \cite{Heusener-Medjerab2014}.

Finally there are deformations of diagonal representations which are not detected by Theorem~\ref{thm:upper-triang-deform}.
The $\SLn[2]$-character variety of $\Gamma_{3,4}$ has four one-dimensional components, three of these components $V_1$, $V_2$ and  $V_3$ contain characters of irreducible representations and one component contains only characters of abelian (reducible) representations. Each component $V_i$ is parametrized by the trace of the meridian, and contains exactly two characters of abelian representations which are limits of irreducible representations (see for example \cite{Klassen1993}).
In order to obtain irreducible $\SLn[3]$ representations of $\Gamma_{3,4}$, we can compose an irreducible  $\SLn[2]$-representation with the canonical representation $r_3\co\SLn[2]\to\SLn[3]$ (see Proposition~3.1 in \cite{Heusener-Medjerab2014}).
In this way we obtain $6$ diagonal representations which are limit of irreducible representations. 
For example, let $\rho_{ab}\co\Gamma_{3,4}\to\SLn[2]$ be given by 
\[
 a\mapsto A = \begin{pmatrix} \eta & 0 \\ 0 & \eta^{-1}\end{pmatrix}\text{ and }
 b\mapsto B = \begin{pmatrix} \zeta & 0 \\ 0 & \zeta^-1\end{pmatrix}
\]
where $ \eta = \exp(i\pi/3)$ is a primitive 6-th root of unity, and 
$\zeta =\exp(i\pi/4)$ is a primitive 8-th root of unity. Then $\eta\zeta^{-1}$ is a primitive 24-th root of unity, and
$r_3\circ\rho_{ab}$ maps the meridian $m=ab^{-1}$ to the diagonal matrix
\footnotesize
\[
 \begin{pmatrix} \lambda_1 & 0 &0 \\ 0 & \lambda_2 & 0\\ 0 & 0 &\lambda_3\end{pmatrix}
=  \begin{pmatrix} (\eta\zeta^{-1})^2 & 0 &0 \\ 0 & 1 & 0\\ 0 & 0 &(\eta\zeta^{-1})^{-2}\end{pmatrix}\,.
\]
\normalsize
Now, $\lambda_1/\lambda_2=\exp(i\pi/6)$ and 
$\lambda_2/\lambda_3 = \lambda_1$ are primitive 12-th roots of unity, and $\lambda_1/\lambda_3 = \lambda_1^2$ is a primitive 6-th root of unity. 
Therefore, all three quotients are roots of $\Delta_{3,4}$, but the deformations of this abelian representation are not detected by Theorem~\ref{thm:upper-triang-deform}. On the other hand, Proposition~3.1 in  \cite{Heusener-Medjerab2014} shows that the representation $r_3\circ\rho_{ab}$ admits irreducible deformations. Moreover, it is possible to prove that the representation $\rho_{ab}$ is contained in a 12-dimensional component of the $\SLn[3]$-representation space $R_3(\Gamma_{3,4})$ (see 
\cite[Proposition 8.2]{Munoz-Porti2016}).

\subsection{The trefoil knot}\label{example:T(3,2)}
By Proposition~\ref{prop:R^n} we know that the representation $\rho_D$ is a smooth point of the two components $R^{(n)}_D$ and $R_{ab}$. For other components this might be not the case.

This situation occurs already for the trefoil knot and representations of its knot group 
$\Gamma$ to $\SLn[3]$. 
It follows from work of J.\ Porti and V.\ Mu\~noz that $R_3(\Gamma)$ has three algebraic components. One component of irreducible representations, one component of totally reducible representations, and one component of partially reducible representations
\cite{Munoz-Porti2016}.

The primitive 6th root of unity $\eta=\exp(i\pi/3)$  is a simple root of the Alexander polynomial $\Delta(t)=t^2-t+1$. Moreover we choose $\lambda$ a primitive 12th root of unity such that $\lambda^2=\eta$.
It follows that
 $\rho_{D(\lambda,\lambda^{-1})}\in R_2(\Gamma)$ is contained in a $4$-dimensional component $R^{(2)}_D\subset R_2(\Gamma)$.

 If $n=3$ and $D=D(\lambda^2,1,\lambda^{-2})$, we obtain two families of partially reducible representations
 $\phi,\psi\co \C^*\times R^{(2)}_D \hookrightarrow R_3(\Gamma)$ given by
 \[
 \phi(s,\varrho) (\gamma) = \big( s^{h(\gamma)} \varrho(\gamma) \big) \oplus s^{-2h(\gamma)}
\quad \text{ and } \quad
 \psi(s,\varrho) (\gamma) = s^{2h(\gamma)} \oplus \big( s^{-h(\gamma)}\varrho(\gamma) \big)\,.
 \]
 The two families are contained in the same component of partially reducible representation of $R_3(\Gamma)$. 
 This component is $9$-dimensional. Notice also that
 $\phi(s,\varrho)$ and $\psi(s^{-1},\varrho)$ are conjugate in $R_3(\Gamma)$.
 
 Now observe that the image of $D\phi(\lambda,\rho_{D_1})$ is contained in $\widetilde V_1$ and that the image of $D\psi(\lambda,\rho_{D_1})$ is contained in $\widetilde V_2$. Hence the tangent space of $\rho_D$ at the variety of partially irreducible representations contains $\widetilde V_1 + \widetilde V_2$ which is $12$-dimensional.
 This implies that $\rho_D$ is a singular point of the component of partially reducible representations.

\section{Local structure of the character variety }\label{sec:character}

Our aim in this section is to study the local structure of the character variety in a neighborhood of the abelian character $\chi_D$. Most results in the literature concern the characters of irreducible representations. Very little is known in the case where the character corresponds to that of an abelian representation.

We will make use of the Luna's \'Etale Slice Theorem which is a tool for the local study of the character variety $X_n(\Gamma)$. 

The general references for this section are \cite{Luna1973}, \cite{Drzet2000LunasST}, \cite{Brion2010}, \cite{popov_vinberg_1994},
\cite{Leila02},  \cite{sikora_character_2012}, \cite{Lubotzky-Magid1985}.

\subsection{A slice  \'etale}\label{sec:slice}
Let $G$ be a reductive algebraic group acting on an affine variety $X$ and let $X\sslash G$ be the quotient variety. 
Let $t \co X \to X\sslash G$ be the quotient morphism.

The following result is of independant importance and appears implicitly in \cite{sikora_character_2012} under some other hypothesis.

\begin{thm}\label{thm:characterslice}
Let $G$ be a reductive algebraic group acting on an affine variety $X$.
Let  $x \in X$ be such that the orbit $O(x)$ is closed in $X$. 

Then there exist a locally closed  subvariety $\mathcal{S}$ of $X$ such that: 
\[
T_x^{Zar}X \cong T_x^{Zar}(O(x))\oplus T_x^{Zar}\mathcal{S}.
\]
Moreover, if $x$ is a smooth point of $X$ then
\[
T_{t(x)}^{Zar}(X\sslash G)\cong T_0\big(T_x^{Zar}\mathcal{S}\sslash G_x\big)
\]
where $G_x$ is the stabilizer of $x$ under the action of $G$. 
\end{thm} 
\begin{proof}
Since the orbit $O(x)$ is closed in $X$, by Luna's \'Etale Slice theorem (\cite{Luna1973}, \cite[Theorem 5.3]{Drzet2000LunasST}), there exists a locally closed subvariety $\mathcal{S}$ of $X$ such that $\mathcal{S}$ is affine and contains $x$, $\mathcal{S}$ is $G_x$ invariant, the image of the $G$-morphism  
$\psi \co G \times_{G_x} \mathcal{S} \to X$ induced by the action of $G$ on $X$ is
a saturated open subset $U$ of $X$.
 Moreover, the restriction of $\psi \co G \times_{G_x} \mathcal{S} \to U$, and
the morphism $\psi\sslash G \co \mathcal{S}\sslash G_x \to U\sslash G$ are \'etale.
In particular, it means that we have the following isomorphisms:
\begin{equation}\label{eq:slice}
T_{(e,x)}^{Zar}\,\psi \co T_{(e,x)}^{Zar}\,G \times_{G_x} \mathcal{S}\stackrel{\cong}\longrightarrow T_x^{Zar} U\cong T_x^{Zar} X
\end{equation}
and
\begin{multline}\label{eq:qslice}
T_{[(e,x)]}^{Zar}\big(\psi\sslash G\big) \co T_{[(e,x)]}^{Zar}\big(G \times_{G_x} \mathcal{S}\sslash G\big)\cong T_{[x]}^{Zar} (S\sslash G_x)\\
\stackrel{\cong}\longrightarrow T_{[x]}^{Zar}(U\sslash G)=T_{[x]}^{Zar}(X\sslash G)\,.
\end{multline}
Moreover, if $x$ is a smooth point of $X$, and using Luna's Slice Theorem at smooth points (\cite{Luna1973}, 
\cite[Theorem 5.4]{Drzet2000LunasST}), we obtain that $T_x^{Zar}X = T_x^{Zar}(O(x))\oplus T_x^{Zar}\mathcal{S}$ and there exists an \'etale $G_x$-invariant morphism 
$\phi\co \mathcal{S}\to T_x^{Zar}\mathcal{S}$ such that $\phi(x) = 0$, $T_x\phi= Id$, the image of $\phi$ is a saturated open subset $W$ of $T_x^{Zar}\mathcal{S}$, the restriction of $\phi\co \mathcal{S} \to W$  and the morphism $\phi\sslash G_x\co \mathcal{S}\sslash G_x \to W\sslash G_x$ are \'etale $G_x$-morphisms.
As a consequence, we have the following isomorphism: 
\begin{equation}\label{eq:qsslice}
T_{[x]}^{Zar}\big(\phi\sslash G_x\big)\co T_{[x]}^{Zar} (S\sslash G_x)\stackrel{\cong}\longrightarrow T_{0}(W\sslash G_x)=T_{0}\big(T_x^{Zar} \mathcal{S}\sslash G_x\big)\,.
\end{equation}
By combining the preceding isomorphisms \eqref{eq:qslice} and \eqref{eq:qsslice}, we obtain the desired isomorphism:
\[
T_{[x]}^{Zar}(X\sslash G)\cong T_0\big(T_x^{Zar}\mathcal{S}\sslash G_x\big)\,.\qedhere
\]
\end{proof}

\subsection{Tangent spaces to slices and cohomology}\label{sec:slicecoho}
The reductive group $G=\SLn$ acts by conjugation on $R_n(\Gamma)$. More precisely, for $A \in\SLn$ and $\rho\in R_n(\Gamma)$ we define $(A\cdot\rho)(\gamma) = A\rho(\gamma)A^{-1}$ for all $\gamma\in\Gamma$.  We will write $\rho\sim\rho^{'}$ if there exists $A \in\SLn$ such that $\rho^{'} = A\cdot \rho$, and we will call $\rho$ and $\rho^{'}$ equivalent. For $\rho\in R_n(\Gamma)$ we define its character $\chi_\rho\co\Gamma\to\C$ by $\chi_\rho(\gamma)=\tr(\rho(\gamma))$. We have $\rho\sim\rho^{'}\Rightarrow \chi_\rho=\chi_{\rho^{'}}$.
By \cite[Theorem 1.28]{Lubotzky-Magid1985}), if $\rho$ and $\rho^{'}$ are semisimple, then  $\rho\sim\rho^{'}$ if and only if $\chi_\rho=\chi_{\rho^{'}}$.

The algebraic quotient or GIT quotient for the action of $\SLn$ on $R_n(\Gamma)$ is called the character variety. This quotient will be denoted by $X_n(\Gamma)=R_n(\Gamma)\sslash\SLn$. The character variety is not necessary an irreducible affine algebraic scheme.
Work of C.~Procesi \cite{Procesi1976} implies that there exists a finite number of group elements $\{\gamma_i\ | \ 1 \leq i \leq M\}\subset\Gamma$ such that the image of $t\co\R_n(\Gamma)\to\C^M$ given by $t(\rho) = (\chi_\rho(\gamma_1), \ldots, \chi_\rho(\gamma_M))$ can be identified with the affine algebraic set $X_n(\Gamma)\cong t(R_n(\Gamma))$.
This justifies the name character variety.

Let $t\co R_n(\Gamma)\to X_n(\Gamma)$; $\rho\mapsto\chi_\rho$ be the canonical projection.
Then $t$ is a regular morphism and for all $\gamma\in\Gamma$, $I_\gamma\co X_n(\Gamma)\to\C$; $\chi_\rho\mapsto\chi_\rho(\gamma)=\operatorname{tr}(\rho(\gamma))$ is a regular morphism. 

For $\rho\in R_n(\Gamma)$, denote the stabilizer of $\rho$ under the $G$-action by conjugation on $R_n(\Gamma)$ by $G_\rho$. It is the centralizer of $\rho(\Gamma)$ in $G$. 

\begin{prop}\label{prop:slice}
Let $\rho$ be a representation in $R_n(\Gamma)$ such that the orbit $O(\rho)$ of $\rho$ is closed in $R_n(\Gamma)$. 

Let $R_\rho$ be an affine subvariety of $R_n(\Gamma)$ which contains $\rho$ and which is invariant under the conjugation action of $G$. Let $\mathcal{S}$ be a slice of $R_\rho$ at $\rho$. 

Then 
\[
T_\rho^{Zar}\mathcal{S}\cong T_\rho^{Zar}R_\rho/T_\rho^{Zar} O(\rho)\quad\hbox{and}\quad T_{\chi_\rho}^{Zar}(R_\rho\sslash G)\cong 
T_{\chi_\rho}^{Zar}(\mathcal{S}\sslash G_\rho)\,.
\]
Moreover, if $\rho$ is a smooth point of $R_\rho$ then
\[
T_{\chi_\rho}^{Zar}(R_\rho\sslash G)\cong T_0\big(T_\rho^{Zar}\mathcal{S}\sslash G_\rho\big)\,.
\]
In particular, if $\rho$ is a reduced smooth point of $R_n(\Gamma)$ such that the orbit $O(\rho)$ of $\rho$ is closed in $R_n(\Gamma)$ then there exists a slice $\mathcal{S}$ of $R_n(\Gamma)$ at $\rho$ such that 
\[
T_\rho^{Zar}\mathcal{S}\cong H^1(\Gamma; {\Ad\rho})\quad\hbox{and}\quad T_{\chi_\rho}^{Zar}X_n(\Gamma)\cong 
T_0\big(H^1(\Gamma; {\Ad\rho})\sslash G_\rho\big)\,.
\]

\end{prop}
The proof of this result is implicitly contained in 
\cite{Leila02} and \cite{sikora_character_2012}. But for the sake of completeness, we will reproduce it here.
\begin{proof}
We have just to prove the first isomorphism. All the other assertions are an immediate consequence of Theorem \ref{thm:characterslice}.

  Let $\psi\co G\times_{G_\rho}\mathcal{S}\to R_\rho$ be the \'etale $G$-morphism given by Luna's Theorem and $U:=Im\psi$. 
We denote by $\theta$ the isomorphism $O(\rho)\stackrel{\theta}\to G/G_\rho$; $\theta(g\rho g^{-1})=[g]$ for $g\in G$.

As the fibration $G\to G/G_\rho$ is locally trivial, the morphism $p\co\ G\times_\rho \mathcal{S}\to G/G_\rho$; $[g,v]\mapsto [g]$ is also a locally trivial fibration with fiber $g\mathcal{S}$ in the \'etale topology (\cite{popov_vinberg_1994}). Let $Z$ be an open set trivializing $G/G_\rho$; then there exists a morphism \'etale $\phi\co Z\times \mathcal{S}\to p^{-1}(Z)$; $([g], \varphi)\mapsto[g,\varphi]$ such that $p\circ\phi=pr_1$ where $pr_1\co Z\times \mathcal{S}\to \mathcal{S}$ is the projection on the first component.

Consider the following chain of morphisms:
\begin{multline*}
T_\rho^{Zar}\mathcal{S}\stackrel{j_2}{\longrightarrow} T_\rho^{Zar}O(\rho)\oplus T_\rho^{Zar}\mathcal{S}\stackrel{d\theta\oplus id}{\longrightarrow} T_{[e]}^{Zar}G/G_\rho \oplus T_\rho^{Zar}\mathcal{S}\stackrel{d\phi}{\longrightarrow} \\
T_{[e,\rho]}^{Zar}(G\times_{G_\rho}\mathcal{S})\stackrel{d\psi}
{\longrightarrow} T_\rho^{Zar}R_\rho\stackrel{q}{\longrightarrow} T_\rho^{Zar}R_\rho/T_\rho^{Zar}O(\rho)
\end{multline*}
where $j_2$ is the natural inclusion and $q$ is the canonical projection. As $p^{-1}(Z)$ is open in $G\times_{G_\rho}\mathcal{S}$ and $\phi$, $\psi$ are \'etale, we obtain that $d\phi$ and $d\psi$ are isomorphisms. Moreover, it is easy to see that $d\psi\circ d\phi\circ(d\theta\oplus id)\circ j_2=di$ where $i\co S\to R_\rho$ is the natural inclusion.
 
On the other hand, let $j_1\co T_\rho^{Zar}O(\rho)\to  T_\rho^{Zar}O(\rho)\oplus T_\rho^{Zar}\mathcal{S}$ be the inclusion map; then 
$d\psi\circ d\phi\circ(d\theta\oplus id)\circ j_1$ is the natural inclusion $T_\rho^{Zar}O(\rho)\to T_\rho^{Zar}R_\rho$. 
As a consequence, $\alpha=q\circ di\co T_\rho^{Zar}S\to T_\rho^{Zar}R_\rho/T_\rho^{Zar}O(\rho)$ is an isomorphism and 
$s=di\circ\alpha^{-1}\co 
T_\rho^{Zar}R_\rho/T_\rho^{Zar}O(\rho)\to T_\rho^{Zar}R_\rho$ 
is a section for $q$. Thus 
$T_\rho^{Zar}\mathcal{S}\cong T_\rho^{Zar}R_\rho/T_\rho^{Zar}O(\rho)$.

\end{proof}

\subsection{Local structure of the character variety}\label{sec:charac}

As the set $R_n^{irr}(\Gamma)$ of irreducible representations in $R_n(\Gamma)$ is Zariski open, its image $X_n^{irr}(\Gamma)=t(R_n^{irr}(\Gamma)) $ is Zariski open in $X_n(\Gamma)$. 
Their Zariski closures in $R_n(\Gamma)$ and $X_n(\Gamma)$ will be denoted by  $\overline{R_n^{irr}(\Gamma)}$ and $\overline{X_n^{irr}(\Gamma)}$. 

\begin{remark} \label{rem:principal_bundle} The action of $\mathrm{PSL}(n)$ on $R_n^{irr}(\Gamma)$ has trivial stabilizers, it follows from the \'etale slice theorem of Luna that the quotient map $R_n^{irr}(\Gamma)\to X_n^{irr}(\Gamma)$ is a principal bundle for the \'etale topologie (see \cite[(1.30)]{Lubotzky-Magid1985}).

\end{remark}

For the proof of the next Proposition, we will make use of the following result which is a generalization of  \cite[Lemma 4.1]{Boyer1998}. We include a proof for completeness.
\begin{lemma}\label{lem:decomposition}
For any algebraic component $X_0 \subset X_n(\Gamma)$ such that $\dim X_0>0$, there is an algebraic component $R_0$ of $t^{-1}(X_0)$ which satisfies $t(R_0) = X_0$. 
 If $X_0$ contains the character of an irreducible representation, then $R_0$ is invariant under conjugation, contains irreducible representations and is uniquely determined by $X_0$. 
\end{lemma}
\begin{proof}
 Let $V_1, V_2,\ldots,V_n$ be the algebraic components of $t^{-1}
(X_0)$. The set $t^{-1}(X_0)$ is closed under conjugation, and so there is a regular map $g \co V_i \times \SLn\to t^{-1}
(X_0)$ defined by $(\rho, g)\to g\rho g^{-1}$. The image of this map contains $V_i$ and since $V_i \times\SLn$ is irreducible, $\overline{g(V_i \times\SLn)}\subset t^{-1}(X_0)$ is also
irreducible. By construction, this is only possible if the image of $g$ is $V_i$. Thus each $V_i$ is closed under conjugation.
According to Theorem 3.5(iv) of \cite{Newstead1978}, the image by $t$ of a closed invariant subset
of $R_n(\Gamma)$ is a Zariski-closed subset of $X_n(\Gamma)$. Thus for each $i$, $0\leq \dim t(V_i)\leq \dim X_0$. Now $t$ is surjective and $\dim X_0>0$, so there is some
index $i$ such that $t(V_i) = X_0$. Taking $R_0$ to be such a $V_i$ proves the first part of the lemma.

Now assume that $X_0$ contains the character of an irreducible representation. After possibly reordering the components $V_1, V_2,\ldots,V_n$, there is some
$k \in \{1, 2,\ldots,n\}$ such that $t(V_i) = X_0$ when $1 \leq i \leq k$, while $\dim t(V_i)<\dim X_0$ when $k + 1 \leq i \leq n$. Suppose that $k \geq 2$ so that in particular $t(V_1) = t(V_2) = X_0$. Now as $X_0$ contains the character of an irreducible representation, we can use Remark~\ref{rem:principal_bundle} to see that $V_1$ and $V_2$ both have
dimension $n^2-1+\dim X_0$. Further $X_0$ contains a Zariski open set $U$ consisting of characters
of irreducible representations \cite[Proposition 27]{sikora_character_2012}. As an irreducible representation
is determined up to conjugation by its character, and both $V_1$ and $V_2$ are
closed under conjugation, it follows that each one contains the $(n^2-1+\dim X_0)$-dimensional set $t^{-1}(U)$. But then $V_1 = V_2$, which is clearly impossible. Hence $k = 1$ and
the lemma follows.
\end{proof}

\begin{prop}\label{prop:smoothchar}
Let $\chi_D$ be the character $\chi_D=\chi(\rho^{(n)}_D)=\chi(\rho_D)$. Denote by $X_n(\Z):=t(R_n(\Z))$ and $X^{(n)}_D:=t(R^{(n)}_D)$  the quotients of the components $R_n(\Z)$ and $R^{(n)}_D$ respectively. Then
\begin{enumerate}
\item\label{sc2} The character $\chi_D$ is a smooth point of each of the two $(n-1)$-dimensional components $X^{(n)}_D$ and $X_n(\Z)$.
   \item\label{sc1} The character $\chi_D$ is a smooth point of $\overline{X_n^{irr}(\Gamma)}$.
\item\label{sc3} The intersection of the two above components at $\chi_D$ is trivial, i.e. 
\[
T_{\chi_D}^{Zar}X_n(\Z)\cap T_{\chi_D}^{Zar}X^{(n)}_D=\{0\}\,.
\]
\end{enumerate}

\end{prop}

\begin{proof}    
{\it(\ref{sc2})}: The orbit $O(\rho_D)$ 
 is closed in $R_n(\Gamma)$ with $\dim O(\rho_D)=n^2-n$ and, 
 as $D=\rho_D(\mu)$ is a regular element of $\SLn$, the stabilizer $G_{\rho_D}$ of $\rho_D$ is the set of diagonal matrices $\mathbb{T}:=(\C^*)^{(n-1)}$.

Now, recall that for the abelian representation $\rho_D$,
$H^1(\Gamma;\Ad\rho_D)$ is
identified  with $\mathbb{C}^{3(n-1)}$ by choosing the basis
represented by the cocycles
$(U_i^+,U^-_i, H_i)\,, 1\leq i \leq n-1$ (see Lemma \ref{elm:basis-cohomology-abelian}).

The $G_{\rho_D}$-action on $Z^1(\Gamma;\Ad\rho_D)$ descends to an action on $H^1(\Gamma;\Ad\rho_D)$. More precisely, the action of a diagonal matrix $\mathrm{diag}(t_1,\ldots,t_n)\in\mathbb{T}$ on a cocycle in $H^1(\Gamma;\Ad\rho_D)$ is given by
\begin{multline*}
(t_1,\ldots,t_n)\cdot\left(U_1^+,U_1^-,\ldots,U_{n-1}^+,U_{n-1}^-, H_1,\ldots,H_{n-1}\right)=\\
\left(\frac{t_1}{t_2} U_1^+,\frac{t_2}{t_1} U_1^-,\ldots,\frac{t_n}{t_{n-1}}U_{n-1}^+,\frac{t_{n-1}}{t_n}U_{n-1}^-, H_1,\ldots,H_{n-1}\right)\,.
\end{multline*}
The action is then trivial on the abelian cocycles $H_i\,, 1\leq i\leq n-1$. 
Furthermore, $\C^*$ acts on $\C^2$ by $\lambda\cdot(z_1,z_2)=(\lambda\,z_1,1/\lambda\,z_2)$ then
$\C^2\sslash\C^*\cong \C$.

In particular, $V_\emptyset\sslash(\C^*)^{n-1}\cong V_\emptyset$ and  $V_{\{1,\ldots,n-1\}}\sslash(\C^*)^{n-1}\cong \C^{n-1}$.
Thus $H^1(\Gamma;\Ad\rho_D)\sslash S_{\rho_D}\cong
\C^{2(n-1)}$ and $\dim T_0\left(H^1(\Gamma;\Ad\rho_D)\sslash S_{\rho_D}\right)=2(n-1)$.

It is a well known fact that $X_n(\Z):=t(R_n(\Z))$ is an irreducible component of $X_n(\Gamma)$ of dimension $n-1$ and that $\chi_D$ is a smooth point of  $X_n(\Z)$ (see \cite[Example 1.2]{dolgachev2003lectures})). Thus $\dim_{\chi_D}X_n(\Z)=n-1$. 

Observe that we can prove this known result by applying  Proposition \ref{prop:slice} to the representation $\rho_D$ as a smooth point of the affine variety $R_n(\Z)$ to obtain a slice $\mathcal{S}_{Ab}$ of $R_n(\Z)$ such that 

\[
T_{\rho_D}^{Zar}\mathcal{S}_{Ab}\cong T_{\rho_D}^{Zar}R_n(\Z)/T_{\rho_D}^{Zar} O(\rho)\cong V_\emptyset
\]
 and 
\begin{align*}
T_{\chi_D}^{Zar}X_n(\Z)& = T_{\chi_D}^{Zar}(R_n(\Z)\sslash\SLn)\\
&\cong T_0\big(T_\rho^{Zar}\mathcal{S}_{Ab}\sslash\mathbb{T}\big)\cong T_0\big(V_\emptyset\sslash(\C^*)^{n-1}\big)\cong V_\emptyset\,.
\end{align*}
So we deduce that $\dim T_{\chi_D}^{Zar}X_n(\Z)=n-1=\dim X_n(\Z)$ and  $\chi_D$ is a smooth point of $X_n(\Z)$.

Again, by applying Proposition \ref{prop:slice} to the representation $\rho_D$ as a smooth point of the affine variety $R^{(n)}_D$, we obtain a slice $\mathcal{S}$ of $R^{(n)}_D$ such that 

\[
T_{\rho_D}^{Zar}\mathcal{S}\cong T_{\rho_D}^{Zar}R^{(n)}_D/T_{\rho_D}^{Zar} O(\rho)\cong V_{\{1,\ldots,n-1\}}
\]
 and 
\begin{align*}
T_{\chi_D}^{Zar}X^{(n)}_D & = T_{\chi_D}^{Zar}(R^{(n)}_D\sslash\SLn)\\
&\cong T_0\big(T_\rho^{Zar}\mathcal{S}\sslash(\C^*)^{n-1}\big)\cong T_0\big(V_{\{1,\ldots,n-1\}}\sslash(\C^*)^{n-1}\big)\cong \C^{n-1}\,.
\end{align*}
So we deduce that $\dim T_{\chi_D}^{Zar}X^{(n)}_D=n-1=\dim X^{(n)}_D$ and thus $\chi_D$ is a smooth point of $X^{(n)}_D$.

{\it(\ref{sc1})}:

 Suppose that $X_0\subset \overline{X_n^{irr}(\Gamma)}$ is an irreducible component which contains $\chi_D$.
By Lemma~\ref{lem:decomposition}, there exists a unique algebraic component
$R_0\subset t^{-1}(X_0)$ which is invariant under conjugation, contains irreducible representations, and such that $t(R_0)=X_0$.
Let $\phi\in R_0$ such that $t(\phi)=\chi_D=t(\rho_D)$. As $R_0$ is invariant under conjugation and closed, we obtain $\overline{O(\phi)}\subset R_0$. By Lemma 1.26 of \cite{Lubotzky-Magid1985}, $\overline{O(\phi)}$ contains a semisimple representation 
$\phi_0$. Thus $t(\phi_0)=t(\phi)=t(\rho_D)$. As $\phi_0$ and $\rho_D$ are semisimple, we conclude that $\phi_0\sim\rho_D$, and hence $\rho_D\in R_0$. 
  By Theorem~\ref{thm:components},  $R^{(n)}_D$ is the only algebraic component of $R_n(\Gamma)$ containing irreducible representations and the representation $\rho_D$. We conclude that $R_0\subset R^{(n)}_D$, and hence 
 $X_0\subset X^{(n)}_D:=t(R^{(n)}_D)$. Since $X_0$ is an algebraic component of $\overline{X_n^{irr}(\Gamma)}$ it follows that $X_0 = X^{(n)}_D$.

{\it(\ref{sc3})}: Note that, as $T_{\chi_D}^{Zar}X_n(\Z)\cong V_\emptyset$, $T_{\chi_D}^{Zar}X^{(n)}_D\cong V_{\{1,\ldots,n-1\}}$ and
$V_\emptyset\cap V_{\{1,\ldots,n-1\}}=\{0\}$, we obtain that 
 the intersection of the two components $X_n(\Z)$ and $X^{(n)}_D$ at $\chi_D$ is trivial.

\end{proof}

There is an analogue of the following theorem when the considered representation is an irreducible one
(\cite[Theorem 2.13]{Lubotzky-Magid1985}).

\begin{thm}\label{thm:proj_to_X(Gamma)}
Let $\operatorname{ad}\co T_{\rho^{(n)}_D}^{Zar}O(\rho^{(n)}_D)\to T_{\rho^{(n)}_D}^{Zar}R_n(\Gamma)$ and $dt\co T_{\rho^{(n)}_D}^{Zar}R_n(\Gamma)\to T_{\chi_D}^{Zar}X_n(\Gamma)$ be the tangent morphisms induced by the conjugation action of $\SLn$ on $\rho^{(n)}_D$ and by the canonical projection $t\co R_n(\Gamma)\to X_n(\Gamma)$.

Then we have the exact sequence
\[
0\to T_{\rho^{(n)}_D}^{Zar}O(\rho^{(n)}_D)\stackrel{ad}{\longrightarrow} T_{\rho^{(n)}_D}^{Zar}R_n(\Gamma) \stackrel{dt}{\longrightarrow} T_{\chi_D}^{Zar}\overline{X_n^{irr}(\Gamma)}\to0
\]
and $H^1(\Gamma;\Ad\rho^{(n)}_D )\cong T_{\chi_D}^{Zar}\overline{X_n^{irr}(\Gamma)}$.

\end{thm}

\begin{proof}
 Consider the commutative diagram:

\[\begin{CD}
T_{\rho^{(n)}_D}^{Zar}O(\rho^{(n)}_D) @>\iota>> \mathrm{Inn}(\Gamma; \Ad\rho^{(n)}_D)=B^1(\Gamma; \Ad\rho^{(n)}_D)\\
@VV\operatorname{ad}V @VVj V\\
T_{\rho^{(n)}_D}^{Zar}R(\Gamma) @>>> \mathrm{Der}(\Gamma; \Ad\rho^{(n)}_D)=Z^1(\Gamma;\Ad\rho^{(n)}_D)
\end{CD}
\]
where $\iota\co\sln\ni a\mapsto \left(\iota_a\co\gamma\mapsto\Ad_{\rho^{(n)}_D(\gamma)}a-a;\ \forall\gamma\in\Gamma\right)$.

Now $j$ is injective and $\iota$ is injective as $H^0(\Gamma; \Ad\rho^{(n)}_D)=0$ (Proposition~\ref{prop:upper-triang-cohom}), so $\operatorname{ad}$ is injective too. By the commutativity of the above diagram and for all $a\in\sln$, we have 
\[
\operatorname{ad}(a)=\displaystyle\frac{d\rho_s}{ds}\big|_{s=0} \text{ where }\rho_s(\gamma)=\left(I+s(\Ad\rho^{(n)}_D(\gamma)a-a)\right)\rho^{(n)}_D(\gamma);\ \text{ for all $\gamma\in\Gamma$.}
\]
Thus, for all $\gamma\in\Gamma$, $dI_\gamma\circ dt\circ \operatorname{ad}(a)=
\displaystyle\frac{d}{ds}\left(\tr\rho_s(\gamma)\right)\big|_{s=0}=\tr\left(a\rho^{(n)}_D(\gamma)-\rho^{(n)}_D(\gamma)a\right)=0$ and $dt\circ \ad(a)=0$.

To prove that $\Ker dt=\operatorname{Im}\,\operatorname{ad}$, we consider a formal path
$\rho_s=\left(I+sU+o(s^2)\right)\rho^{(n)}_D$ such that $U=\displaystyle\frac{d\rho_s}{ds}\big|_{s=0}\in\Ker dt$.
The relation $dt\left(\displaystyle\frac{d\rho_s}{ds}\big|_{s=0} \right)=0$ implies that 
$\displaystyle\frac{d}{ds}(\tr\rho_s)\big|_{s=0}=0.$ 

By Corollary~\ref{cor:H^1sln} and modulo conjugation by a coboundary, we may assume that there exists complexes $\epsilon_i\in\C$, and cochains $a_i\in C^1(\Gamma;\C)$, $1\leq i\leq n-1$, such that the cocycle $U\in Z^1(\Gamma;\Ad\rho^{(n)}_D)$ is of the form:

\[
U =
\begin{pmatrix}
a_1 & * & * & \dots & * \\
\epsilon_1u_1^- & a_2 & * & \ddots &\vdots \\
0&\epsilon_2u_2^- & a_3  & \ddots &\vdots \\
\vdots & \ddots & \ddots & \ddots &* \\
0 & \dots & 0  & \epsilon_{n-1}u_{n-1}^- & a_n
\end{pmatrix}\in Z^1(\Gamma; \Ad\rho^{(n)}_D)
\,.
\]

Thus $0=\displaystyle\frac{d}{ds}(\tr\rho_s(\gamma))\big|_{s=0}$
implies that: 
\[
 \sum_{i=1}^n\lambda_i^{h(\gamma)}a_i(\gamma)+
 \sum_{i=1}^{n-1}\lambda_{i+1}^{h(\gamma)}\epsilon_iu_i^-
(\gamma)u_i^+(\gamma)=0\quad \forall\gamma\in\Gamma
 \,.
\]

As a consequence, we obtain:  
\[
\sum_{i=1}^na_i(\gamma)+
 \sum_{i=1}^{n-1}\epsilon_iu_i^-(\gamma)u_i^+(\gamma)=0\quad \forall\gamma\in\Gamma^{'}\,.
 \]

Now $\sum_{i=1}^na_i=0$ as $U\in Z^1(\Gamma;\Ad\rho^{(n)}_D)$, so we obtain: 
\[\sum_{i=1}^{n-1}\epsilon_iu_i^-(\gamma)u_i^+(\gamma)=0 \quad \forall\gamma\in\Gamma^{'}\,.
\]
The fact that the derivations $u_i^\pm$, $1\leq i\leq n-1$, factor through $\Gamma^{''}$ give the equation: 
\[
\sum_{i=1}^{n-1}\epsilon_iu_i^-(x)u_i^+(x)=0\quad  \forall x\in\Gamma^{'}/\Gamma^{''}\otimes\C\,.
\]
Recall that as $\alpha_i=\lambda_i/\lambda_{i+1}$, $1\leq i\leq n-1$ are simple roots of the Alexander polynomial $\Delta_K$ and $\Delta_K(\lambda_i/\lambda_{j})\not=0$, $|i-j|\geq2$, the Alexander module decomposes as 
$\displaystyle\oplus_{i=1}^{n-1}\left( \Lambda/t-\alpha_i\oplus  \Lambda/t-\alpha_i^{-1} \right)\oplus \tau$ where $\tau$ is a torsion $\Lambda$-module with no $(t-\alpha_i^{\pm 1})$-torsion. 
Choose a basis $(e_1^+, e_1^-,\ldots,e_{n-1}^+, e_{n-1}^-)$ of $\oplus_{i=1}^{n-1}\left( \Lambda/(t-\alpha_i)\oplus  \Lambda/(t-\alpha_i^{-1}) \right)$. 

Modulo coboundaries, we may suppose that the non principal derivations $\tilde u_i^\pm$
are generators of $H^1(\Gamma;\C_{\alpha_i^{\pm 1}})$ as described in the proof of Proposition~\ref{prop:H1GammaC}. In particular,
\[
\tilde u_i^\pm(e_i^\pm)=1; \ \tilde u_i^\pm(e_i^\mp)=0;\ \tilde u_i^\pm(e_j^\pm)=
\tilde u_i^\pm(e_j^\mp)=0; \ 1\leq i\not=j\leq n-1\;.
\]

Evaluating the preceeding equation at $x=e_i^++e_i^-$, $1\leq i\leq n-1$, gives:
\begin{align*}
0&=\sum_{j=1}^{n-1}\epsilon_j\tilde u_j^-(e_i^++e_i^-)\tilde u_j^+(e_i^++e_i^-)\\
&=\epsilon_i\tilde u_i^-(e_i^++e_i^-)\tilde u_i^+(e_i^++e_i^-)\\
&=\epsilon_i\tilde u_i^-(e_i^-)\tilde u_i^+(e_i^+)\\
&=\epsilon_i\,.
\end{align*}
We obtain that $\epsilon_i=0$ for all $1\leq i\leq n-1$ and the cocycle 
$U\in Z^1(\Gamma;\Ad\rho^{(n)}_D)$ is thus of the form:

\[
U =
\begin{pmatrix}
a_1 &  * & \dots & * \\
0 & a_2  & \ddots &\vdots \\
\vdots &  \ddots & \ddots &* \\
0 & \dots   & 0 &  a_n
\end{pmatrix}\in Z^1(\Gamma;C_0)
\,.
\]
By Remark~\ref{rem:trace-cohom} the traceless cohomolgy group $H^1_0(\Gamma; C_0/C_2)\cong0$, thus the cocycle $U\in\Ker dt$ is in fact a coboundary and $\Ker dt=
\operatorname{Im}\,\operatorname{ad}$.

Finally, the surjectivity of the morphism $dt$ is a consequence of the fact that 
\begin{align*}
\operatorname{rg}(dt)&=\dim T_{\rho^{(n)}_D}^{Zar}R_n(\Gamma)-\operatorname{rg}(\operatorname{ad})=n^2+n-2-(n^2-1)\\
&=n-1=\dim T_{\chi_D}^{Zar}X^{(n)}_D=
\dim T_{\chi_D}^{Zar}\overline{X_n^{irr}(\Gamma)}\,.\qedhere
\end{align*}
\end{proof}

\bibliographystyle{plain}
\bibliography{NNewVersion}

\begin{thebibliography}{10}

\bibitem{BenAbdelghani-Lines2002}
Leila~Ben Abdelghani and Daniel Lines.
\newblock Involutions on knot groups and varieties of representations in a
  {L}ie group.
\newblock {\em J. Knot Theory Ramifications}, 11(1):81--104, 2002.

\bibitem{BenAbdelghani2000}
Leila Ben~Abdelghani.
\newblock Espace des repr\'esentations du groupe d'un n\oe ud classique dans un
  groupe de {L}ie.
\newblock {\em Ann. Inst. Fourier (Grenoble)}, 50(4):1297--1321, 2000.

\bibitem{Leila02}
Leila Ben~Abdelghani.
\newblock Vari\'et\'e des caract\`eres et slice \'etale de l'espace des
  repr\'esentations d'un groupe.
\newblock {\em Ann. Fac. Sci. Toulouse Math. (6)}, 11(1):19--32, 2002.

\bibitem{BenAH15}
Leila Ben~Abdelghani and Michael Heusener.
\newblock Irreducible representations of knot groups into
  $\mathrm{SL}(n,\mathbf{C})$.
\newblock {\em Publ. Mat.}, 61(2):751--764, 2017.

\bibitem{BenAbdelghani-Heusener-Jebali2010}
Leila Ben~Abdelghani, Michael Heusener, and Hajer Jebali.
\newblock Deformations of metabelian representations of knot groups into {${\rm
  SL}(3,{\bf C})$}.
\newblock {\em J. Knot Theory Ramifications}, 19(3):385--404, 2010.

\bibitem{Bourbaki}
Nicolas Bourbaki.
\newblock {\em Algebra {II}. {C}hapters 4--7}.
\newblock Elements of Mathematics (Berlin). Springer-Verlag, Berlin, 2003.

\bibitem{Boyer1998}
Zhang~X. Boyer, S.
\newblock On culler-shalen seminorms and dehn filling.
\newblock {\em Annals of Mathematics. Second Series}, 148(3):737--801, 1998.

\bibitem{Brion2010}
Michel Brion.
\newblock Introduction to actions of algebraic groups.
\newblock {\em Les cours du CIRM}, 1(1):1--22, 2010.

\bibitem{Brown1982}
Kenneth~S. Brown.
\newblock {\em Cohomology of groups}, volume~87 of {\em Graduate Texts in
  Mathematics}.
\newblock Springer-Verlag, New York, 1982.

\bibitem{Burde}
Gerhard Burde.
\newblock Darstellungen von {K}notengruppen.
\newblock {\em Math. Ann.}, 173:24--33, 1967.

\bibitem{deRham}
Georges de~Rham.
\newblock Introduction aux polyn\^{o}mes d'un n\oe ud.
\newblock {\em Enseign. Math. (2)}, 13:187--194 (1968), 1967.

\bibitem{dolgachev2003lectures}
Igor Dolgachev.
\newblock {\em Lectures on invariant theory}, volume 296 of {\em London
  Mathematical Society Lecture Note Series}.
\newblock Cambridge University Press, Cambridge, 2003.

\bibitem{Drzet2000LunasST}
Jean-Marc Dr\'{e}zet.
\newblock Luna's slice theorem and applications.
\newblock In {\em Algebraic group actions and quotients}, pages 39--89. Hindawi
  Publ. Corp., Cairo, 2004.

\bibitem{EisenbudHarris}
David Eisenbud and Joe Harris.
\newblock {\em The geometry of schemes}, volume 197 of {\em Graduate Texts in
  Mathematics}.
\newblock Springer-Verlag, New York, 2000.

\bibitem{Goldman1984}
William~M. Goldman.
\newblock The symplectic nature of fundamental groups of surfaces.
\newblock {\em Adv. in Math.}, 54(2):200--225, 1984.

\bibitem{Goldman1985}
William~M. Goldman.
\newblock Representations of fundamental groups of surfaces.
\newblock In {\em Geometry and topology ({C}ollege {P}ark, {M}d., 1983/84)},
  volume 1167 of {\em LNM}, pages 95--117. Springer, Berlin, 1985.

\bibitem{GoldmanMillson1988}
William~M Goldman and John~J Millson.
\newblock The deformation theory of representations of fundamental groups of
  compact k{\"a}hler manifolds.
\newblock {\em Publ. Math. Inst. Hautes Etudes Sci.}, (67):43--96, 1988.

\bibitem{Gordon}
C.~McA. Gordon.
\newblock Some aspects of classical knot theory.
\newblock In {\em Knot theory ({P}roc. {S}em., {P}lans-sur-{B}ex, 1977)},
  volume 685 of {\em Lecture Notes in Math.}, pages 1--60. Springer, Berlin,
  1978.

\bibitem{HerzogHibi}
J\"urgen Herzog and Takayuki Hibi.
\newblock {\em Monomial ideals}, volume 260 of {\em Graduate Texts in
  Mathematics}.
\newblock Springer-Verlag London, Ltd., London, 2011.

\bibitem{Heusener2016}
Michael Heusener.
\newblock {${\rm SL}(n,\bold C)$}-representation spaces of knot groups.
\newblock In {\em Topology, Geometry and Algebra of low-dimensional manifolds},
  pages 1--26. RIMS K\^oky\^uroku, 2016.
\newblock
  \url{http://www.kurims.kyoto-u.ac.jp/~kyodo/kokyuroku/contents/1991.html}.

\bibitem{Heusener-Medjerab2014}
Michael Heusener and Ouardia Medjerab.
\newblock Deformations of reducible representations of knot groups into
  {$\text{SL}(n,\text{C})$}.
\newblock {\em Math. Slovaca}, 66(5):1091--1104, 2016.

\bibitem{Heusener-Porti2005}
Michael Heusener and Joan Porti.
\newblock Deformations of reducible representations of 3-manifold groups into
  {${\rm PSL}\sb 2(\Bbb C)$}.
\newblock {\em Algebr. Geom. Topol.}, 5:965--997, 2005.

\bibitem{Heusener-Porti2015}
Michael Heusener and Joan Porti.
\newblock Representations of knot groups into $\mathrm{SL}_n(\mathbf{C})$ and
  twisted {A}lexander polynomials.
\newblock {\em Pacific J. Math.}, 277:313--354, 09 2015.

\bibitem{HP23}
Michael Heusener and Joan Porti.
\newblock The scheme of characters in $\mathrm{SL}_2$.
\newblock {\em Trans. Amer. Math. Soc.}, 376:6283--6313, 2023.

\bibitem{Heusener-Porti-Suarez2001}
Michael Heusener, Joan Porti, and Eva Su{\'a}rez~Peir{\'o}.
\newblock Deformations of reducible representations of 3-manifold groups into
  {${\rm SL}\sb 2(\bold C)$}.
\newblock {\em J. Reine Angew. Math.}, 530:191--227, 2001.

\bibitem{Klassen1993}
Eric~Paul Klassen.
\newblock Representations in $\mathrm{SU}(2)$ and the fundamental group of the
  {W}hitehead link and of double knots.
\newblock {\em Forum Math.}, 5:93--109, 1993.

\bibitem{Kunz}
Ernst Kunz.
\newblock {\em Introduction to commutative algebra and algebraic geometry}.
\newblock Modern Birkh\"auser Classics. Birkh\"auser/Springer, New York, 2013.

\bibitem{Lubotzky-Magid1985}
Alexander Lubotzky and Andy~R. Magid.
\newblock Varieties of representations of finitely generated groups.
\newblock {\em Mem. Amer. Math. Soc.}, 58(336):xi+117, 1985.

\bibitem{Luna1973}
Domingo Luna.
\newblock Slices \'{e}tales.
\newblock In {\em Sur les groupes alg\'{e}briques}, Bull. Soc. Math. France,
  M\'{e}m. 33, pages 81--105. Soc. Math. France, Paris, 1973.

\bibitem{Milne}
J.~S. Milne.
\newblock {\em Algebraic groups}, volume 170 of {\em Cambridge Studies in
  Advanced Mathematics}.
\newblock Cambridge University Press, Cambridge, 2017.

\bibitem{Mumford-redBook}
David Mumford.
\newblock {\em The red book of varieties and schemes}, volume 1358 of {\em
  LNM}.
\newblock Springer-Verlag, Berlin, 1988.

\bibitem{Munoz-Porti2016}
Vicente Mu{\~n}oz and Joan Porti.
\newblock Geometry of the $\mathrm{SL}(3,\mathbb{C})$--character variety of
  torus knots.
\newblock {\em AGT}, 16(1):397--426, Feb 2016.

\bibitem{Newstead1978}
P.~E. Newstead.
\newblock {\em Introduction to moduli problems and orbit spaces}, volume~51 of
  {\em Tata Institute of Fundamental Research Lectures on Mathematics and
  Physics}.
\newblock Tata Institute of Fundamental Research, Bombay, 1978.

\bibitem{popov_vinberg_1994}
V.L. Popov and E.~B. Vinberg.
\newblock Invariant theory.
\newblock In A.~N. Parshin and I.~R. Shafarevich, editors, {\em Algebraic
  geometry. {IV}}, volume~55 of {\em EMS}, pages vi+284. Springer-Verlag,
  Berlin, 1994.

\bibitem{Procesi1976}
C.~Procesi.
\newblock The invariant theory of {$n\times n$} matrices.
\newblock {\em Advances in Math.}, 19(3):306--381, 1976.

\bibitem{sikora_character_2012}
Adam Sikora.
\newblock Character varieties.
\newblock {\em Trans. Amer. Math. Soc.}, 364(10):5173--5208, 2012.

\bibitem{Weil1964}
Andr{\'e} Weil.
\newblock Remarks on the cohomology of groups.
\newblock {\em Ann. of Math. (2)}, 80:149--157, 1964.

\end{thebibliography}

\end{document}